\documentclass[final]{siamltex}
\usepackage{graphics}
\usepackage{amsmath,amssymb,amsfonts}
\usepackage{mathrsfs}
\usepackage{wasysym}
\usepackage{algorithm}
\usepackage{algorithmic}
\usepackage{multirow,booktabs,comment}

\usepackage{todonotes}
\usepackage[normalem]{ulem}
\newtheorem{remark}{\textbf{Remark}}[section]

\usepackage{geometry}
\usepackage[colorlinks=true,urlcolor=blue,citecolor=red,
anchorcolor=black,linkcolor=blue]{hyperref}

\newtheorem{assumption}[theorem]{Assumption}

\newtheorem{property}[theorem]{Property}

\usepackage{lineno}

\newcommand{\cA}{\mathcal{A}}

\newcommand{\com}[1]{ }

\begin{document}
\title{Convergence analysis of a finite element approximation of
minimum action methods}
\author{Xiaoliang Wan%
\thanks{Department of Mathematics and Center for Computation
\& Technology, Louisiana State University, Baton Rouge, 70803 (\tt xlwan@math.lsu.edu, jzhai1@lsu.edu).}
\and
Haijun Yu%
\thanks{School of Mathematical Sciences, University of Chinese Academy of Sciences, Beijing  100049, China;
	NCMIS \& LSEC, Institute of Computational Mathematics and Scientific/Engineering Computing, Academy of Mathematics and Systems Science, Beijing 100190, China (\tt hyu@lsec.cc.ac.cn)}
\and
Jiayu Zhai\footnotemark[1]
}

\maketitle

\begin{abstract}
In this work, we address the convergence of a finite element
approximation of the minimizer of the Freidlin-Wentzell (F-W) action
functional for non-gradient dynamical systems perturbed by small
noise. 
The F-W theory of large deviations is a rigorous
mathematical tool to study small-noise-induced transitions in a
dynamical system. The central task in the application of F-W
theory of large deviations is to seek the minimizer and minimum
of the F-W action functional. We discretize the F-W action
functional using linear finite elements, and establish the
convergence of {the approximation} through $\Gamma$-convergence.
\end{abstract}



\section{Introduction}
We consider a general dynamical system perturbed by small noise
\begin{equation}\label{eqn:sode}
dX=b(X)\,dt+\sqrt{\varepsilon}\,dW(t),
\end{equation}
where $\varepsilon$ is a small positive number and $W(t)$ is a
standard Wiener process in $\mathbb{R}^n$. The long-term
behavior of the perturbed system is characterized by the
small-noise-induced transitions between the equilibriums of the
unperturbed system 
\begin{equation}\label{eqn:dynamical_system}
\frac{dx}{dt}=b(x),\quad x\in \mathbb{R}^n.
\end{equation}
These transitions rarely occur but have a
major impact. This model can describe many critical phenomena in
physical, chemical, and biological systems, such as
non-equilibrium interface growth \cite{Fogedby_PRE2009,Smith_PRE17},
regime change in
climate \cite{Yao_JCP15}, switching in biophysical network
\cite{Wells_PRX15}, hydrodynamic instability
\cite{Wan_MAM_Poiseuille,Wan_MAM_dconsistent}, etc.

The Freidlin-Wentzell (F-W) theory of large deviations provides
a rigorous mathematical framework to understand the small-noise-induced
transitions in general dynamical systems, where the key object
is the F-W action functional, and the critical quantities include
the minimizer and minimum of the F-W action functional \cite{FW}. Starting from
\cite{String2002}, the large deviation principle given by the F-W theory
has been approximated numerically, especially for non-gradient systems,
and the numerical methods are, in general, called minimum action method (MAM).
More specifically, the following optimization problems need to be addressed:
\begin{equation} \label{ProbI}
\textrm{Problem I}:\quad S_T(\phi^*)=\inf_{\substack{\phi(0)=x_1,\\
\phi(T)=x_2}}S_T(\phi),
\end{equation}and
\begin{equation} \label{ProbII}
\textrm{Problem II}:\quad S_{T^*}
(\phi^*)=\inf_{T\in\mathbb{R}^+}\inf_{\substack{\phi(0)=x_1,\\
\phi(T)=x_2}}S_T(\phi),
\end{equation}
where 
\begin{equation}\label{eqn:action}
S_T(\phi)=\frac{1}{2}\int_{0}^T
|{\phi}'-b(\phi)|^2\,dt
\end{equation}
is called the action functional. 
Here $\phi(t)$ is a path connecting $x_1$ and $x_2$
in the phase space on the time interval $[0,T]$. The {minima} and minimizers of
Problems I and II characterize the difficulty of the {small-}noise-induced
transition from $x_1$ to the vicinity of $x_2$, see equations
\eqref{eqn:LDP-fixed-T} and \eqref{eqn:LDP-quasi-potential}.
In Problem I, the transition is restricted to a certain time scale
$T$, which is relaxed in Problem II. Let $\phi^*(t)$ be the
minimizer of either Problem I or Problem II, which is also called
the minimal action path (MAP), or the instanton in physical
literature related to path integral. For problem II, we have an
optimal integration time $T^*$ which can be either finite or infinite
depending on the states $x_1$ and $x_2$.

We will focus on the minimum action method for non-gradient systems.
For gradient systems, the minimal action path is consistent with the
minimum energy path, and the counterpart version of minimum action method
includes string method \cite{String2002}, nudged elastic band method
\cite{NEB1998}, etc.,
which takes advantage of the property that the minimal action path is
parallel to the drift term of the stochastic differential equation.
For non-gradient systems, this property does not hold and
a direct optimization of the F-W action functional needs to be considered.
The main numerical difficulty comes from the separation of slow
dynamics around critical points from fast dynamics elsewhere. More
specifically, the MAP will be mainly captured by the fast dynamics subject
to a finite time, but it will take infinite time to pass a critical point.
To overcome this difficulty, there exist two basic techniques: (1)
non-uniform temporal discretization, and (2) reformulation of the
action functional with respect to arc length. Two typical techniques
to achieve non-uniform temporal discretization include moving mesh
technique and adaptive finite element method. The moving mesh technique
starts from a uniform finite mesh and {redistributes} the grid points iteratively
such that more grids are assigned into the region of fast dynamics and
less grids into the region of slow dynamics. This technique is used
by the adaptive minimum action method (aMAM)
\cite{Zhou_JCP08,Wan_MAM11,Wan_MAM_parallel,Zhou_2016}.
The adaptive finite element method starts from a coarse mesh and has an
inclination to refine the mesh located in the region of fast dynamics
\cite{Wan_tMAM,Wan_tMAM_hp}. The main difference of these
two techniques from the efficiency point of view is that the moving mesh
technique needs a projection from fine mesh to fine mesh, i.e., global
reparameterization, while the adaptive finite {element} method only needs local
projection in the elements that have been refined. To eliminate the scale
separation from dynamics, one can consider parameterization
of the curves geometrically, i.e., a change of variable from time to
arc length, which is used in the geometric minimum action method (gMAM)
\cite{Heymann_CPAM08,Grafke_MMS14,Grafke_2016}. The change of variable
induces two difficulties.
One is related to accuracy and the other one is related to efficiency.
The mapping from time to arc length is nonlinear and the {Jacobian} 
{of the transform between time and arc length variables} is
singular around critical points since an infinite time domain has been
mapped to a finite arc length. Unknown critical points along the minimal
action path may deteriorate the approximation accuracy unless they can
be identified accurately. To use arc length for parameterization, {we have that} the
velocity {is} a constant, which means in each iteration step a
global reparameterization is needed to maintain this constraint.

Both aMAM and gMAM target to the case that $T^*=\infty$. In aMAM, a finite
but large $T$ is used while in gMAM, the infinite $T^*$ is mapped to a finite
arc length. So, aMAM is not able to deal with Problem II subject to a finite
$T^*$ since a fixed $T$ is required while gMAM is not able to deal with
Problem I since $T$ has been removed. To deal with both Problem I and II
in a more consistent way, we have developed {a minimum action method 
with optimal linear time scaling (temporal minimum action method, or tMAM)} 
{coupled} with adaptive
finite element discretization \cite{Wan_tMAM,Wan_tMAM_hp}. The method is
based on two observations: (1)
for any given transition path, there exists a unique $T$ to minimize
the action functional subject to a linear scaling of time, and (2) for
transition paths defined on a finite element approximation space, the
optimal integration time $T^*$ is always finite but increases as the
approximation space is refined. The first observation removes the parameter
$T$ in Problem II and the second observation guarantees that the discrete
problem of Problem II is well-posed after $T$ is removed. Problem I becomes
a special case of our reformation of Problem II. This way, tMAM is able to
deal with both Problem I and II.

Although many techniques have been developed from the algorithm point of
view, few numerical analysis has been done for minimum action method. We
want to fill this gap partially in this work. We consider a general
stochastic {ordinary differential equation (ODE)} system \eqref{eqn:sode}. 
The discrete action functional
$S_{T,h}$ will be given by linear finite elements for simplicity, where
$h$ indicates the element size. Due to the general assumption
for $b(x)$, we will focus on the convergence of the minimizers of $S_{T,h}$ 
as $h\rightarrow0$ and only provide a priori error estimate for the approximate
solution when $b(x)$ is a linear {symmetric positive definite (SPD)} system. 
For Problem I, the
convergence of the minimizer $\phi^*_h$ to $\phi^*$ is established by the
$\Gamma$-convergence of the discretized action functional. For Problem II,
we employ and analyze the strategy developed in \cite{Wan_tMAM} to deal with
the optimization with respect to $T$. More specifically, we reformulate
the problem from $[0,T]$ to $[0,1]$ by a linear time scaling $s=t/T$ and
replace the integration time $T$ with a functional $\hat{T}(\bar{\phi})$
with $\bar{\phi}(s)=\phi(t/\hat{T})$, where $\hat{T}(\bar{\phi})$ is the optimal
integration time for a given transition path $\bar{\phi}$. When $T^*$ is finite,
the convergence of the minimizer $\bar{\phi}_h^*$ to
$\bar{\phi}^*(s)=\phi^*(t/T^*)$ can be established by the
$\Gamma$-convergence of the discretized action functional. When $T^*=\infty$,
the linear mapping from $t$ to $s$ does not hold. We demonstrate that
the sequence $\{\bar{\phi}_h^*\}$ still provides a minimizing
sequence as $h\rightarrow0$ and establish the convergence using the
results from gMAM. Due to the nonlinearity of $b(x)$, the Euler-Lagrange {(E-L)}
equation associated with the action functional is, in general, a
nonlinear elliptic problem for Problem I. For problem II subject to a
optimal linear time scaling, the {E-L} equation remains the same
form as Problem I with the parameter $T$ being replaced by a functional
$\hat{T}(\bar{\phi})$, which becomes a nonlocal and nonlinear elliptic equation.
When $b(x)$ is a linear {SPD} system, we are able to establish the a priori error
estimate for $\bar{\phi}_h^*$, where the {E-L} equation
is {a} nonlocal and nonlinear elliptic problem of Kirchhoff type.

The remain part of this paper is organized as follows. In Section 2, we describe
the problem setting. A reformulation of the Freidlin-Wentzell action
functional is given in Section 3 to deal with
the optimization with respect to $T$ in Problem II.
We establish the convergence of finite
element approximation in Section 4 for general {stochastic ODE} systems. In
Section 5, we apply our method to a linear {stochastic ODE} system and
provide a prior error estimate of the approximation
solution. Numerical illustrations are given in Section 6
followed by a summary section.

\section{Problem description}
We consider the small-noise-perturbed dynamical system  \eqref{eqn:sode}. 
Let $x_1$ and $x_2$ be two arbitrary points in the phase space. The
Freidlin-Wentzell theory of large deviations provides asymptotic results
to estimate the transition probability from $x_1$ to the vicinity of
$x_2$ when $\varepsilon\rightarrow0$. If we restrict the transition on
a certain time interval $[0,T]$, we have
\begin{equation}\label{eqn:LDP-fixed-T}
\lim_{\delta\downarrow0}
\lim_{\varepsilon\downarrow0}-\varepsilon\log\Pr(\tau_{\delta}\leq
T)=\inf_{\substack{\phi(0)=x_1,\\ \phi(T)=x_2}}S_T(\phi),
\end{equation}
where $\tau_\delta$ is the first entrance time of the $\delta$-neighborhood
of $x_2$ for the random process $X(t)$ starting from $x_1$. {The path variable $\phi$ connecting $x_0$ and $x_1$, over which the action functional is minimized, is called a \textit{transition path}.} If the time scale
is not specified, the transition probability can be described with respect to
the quasi-potential from $x_1$ to $x_2$:
\begin{equation}\label{eqn:quasi-potential}
V(x_1,x_2){:=}\inf_{T\in\mathbb{R}^+}\inf_{\substack{\phi(0)=x_1,\\
\phi(T)=x_2}}S_T(\phi).
\end{equation}
The probability meaning of $V(x_1,x_2)$ is
\begin{equation}\label{eqn:LDP-quasi-potential}
V(x_1,x_2)={\inf_{T\in\mathbb{R}^+}}\lim_{\delta\downarrow0}
\lim_{\varepsilon\downarrow0}-\varepsilon\log\Pr(\tau_{\delta}\leq T).
\end{equation}
We in general call the asymptotic results given in equations
\eqref{eqn:LDP-fixed-T} and \eqref{eqn:LDP-quasi-potential} large deviation
principle {(LDP)}. We use $\phi^*$ to indicate transition path that minimizes the
action functional in equation \eqref{eqn:LDP-fixed-T} or
\eqref{eqn:LDP-quasi-potential}, which is also called the minimal action
path (MAP) \cite{E_CPAM04}. The MAP $\phi^*$ is the most probable transition
path from $x_1$ to $x_2$. For the quasi-potential, we let $T^*$ indicate the
optimal integration time, which can be either finite or infinite depending
on $x_1$ and $x_2$. The importance of LDP is that it simplifies the
computation of transition probability, which is a path integral in a
function space, to seeking the minimizers $\phi^*$ or $(T^*,\phi^*)$. 
From the application point of view, one central task of Freidlin-Wentzell
theory of large deviations is then to solve the Problem I and Problem II defined in \eqref{ProbI} and \eqref{ProbII}, correspondingly. 
For Problem II, we need to optimize the action functional with respect to
the integration time $T$. We will present a reformulation of $S_T$ in 
{Section} \ref{sec:reformulation_ST} to deal with this case.

To analyze the convergence properties of numerical approximations for Problem I and II, we need some assumptions on $b(x)$.
\begin{assumption}\label{assump:1}
	\begin{enumerate}
		\item[(1)] $b(x)$ is Lipschitz continuous in a big ball, i.e., there
		exist constants $K>0$ and $R_1>0$, such that
		\begin{equation}\label{eq:LLipCont}
		|b(x)-b(y)|\leq K|x-y|, \quad \forall\ x,y\in B_{R_1}(0),
		\end{equation}
		where $|\cdot|$ {denotes} the $\ell_2$ norm
		of a vector in $\mathbb{R}^n$;
		\item[(2)] There exist positive numbers $\beta, R_2$, such that
		\begin{equation}\label{eq:binfcond}
		\langle b(x), x\rangle \le -\beta |x|^2,\quad \forall\ |x|\ge R_2,
		\end{equation}
		where	$R_2^2 \le R_1^2 - \tfrac{S^*}{\beta}$,
		and 
		\[
		S^* = \max_{x,y\in B_{R_2}(0)}\frac{1}{2}\int_0^1
		\big|y-x - b\big(x+(y-x)t\big)\big|^2\,dt.
		\]
		\item[(3)] The solution points of $b(x)=0$ are isolated.
	\end{enumerate}
\end{assumption}
\begin{lemma}\label{lem:for_assump_2}
	Let assumption \eqref{eq:binfcond} hold. If both the 
	starting and ending points of a MAP $\phi(t)$ are inside $B_{R_2}(0)$, 
	then $\phi(t)$ is located within $B_{R_1}(0)$ for any $t$. 
\end{lemma}
\begin{proof}
	Suppose that $\phi(t)$ is a MAP outside of $B_{R_2}(0)$ but connecting 
	two points $x$ and $y$ on the the surface of $B_{R_2}(0)$. 
	Let $w(t) = \phi' - b(\phi)$. We have
	\begin{equation}
	\phi' = b(\phi) + w.
	\end{equation}
	Taking inner product on both sides of the above equation with $2\phi$,
	we get
	\begin{equation}
	\frac{d|\phi|^2}{dt} = 2 \langle b(\phi),\phi \rangle + 2 \langle w,
	\phi\rangle.
	\end{equation}
	Then by using Cauchy's inequality with $\beta$, and assumption
	\eqref{eq:binfcond}, we get
	\begin{equation}
	\frac{d|\phi|^2}{dt} \le -2\beta |\phi|^2 + \frac{1}{2\beta} |w|^2
	+ 2\beta |\phi|^2 = \frac{1}{2\beta} |w|^2.
	\end{equation}
	Taking integration, and using the definition of minimum action,
	we obtain a bound for any $t$ along the MAP:
	\begin{equation}
	|\phi|^2
	\le |x|^2 + \int_0^t \frac{1}{2\beta} |w|^2 \,dt
	\le R_2^2 + \frac{1}{\beta} S_{T^*}(\phi)
	\le R_2^2 + \frac{1}{\beta} S^*\leq R_1^2,
	\end{equation}
	which means that the whole MAP is located within $B_{R_1}(0)$. 
\end{proof}

\begin{remark}\label{rmk:assump_2}
The assumptions \eqref{eq:LLipCont} and \eqref{eq:binfcond} allow most of
the physically relevant smooth nonlinear dynamics. It is seen from Lemma 
\ref{lem:for_assump_2} that the second assumption
\eqref{eq:binfcond} is used to restrict all MAPs of interest inside
$B_{R_1}(0)$. 
For simplicity and without loss of generality, we will assume from now on 
that the Lipschitz continuity of $b(x)$ is global, namely, $R_1=\infty$.
For the general case given in Assumption \ref{assump:1}, one can
achieve all the conclusions by restricting the theorems and their
proofs into $B_{R_1}(0)$. 
\end{remark}

We now summarize some notations that will be used later on. For
$\phi(t)\in\mathbb{R}^n$ defined on $\Gamma_T=[0,T]$, we let
$|\phi|^2=\sum_{i=1}^n|\phi_i|^2$ and
$|\phi|^2_{m,\Gamma_T}=\sum_{i=1}^n|\phi_i|_{m,\Gamma_T}^2$,
where $\phi_i$ is the $i$-th component of $\phi$ and $|\phi_i|^2_{m,\Gamma_T}=\int_{\Gamma_T}|\phi_i^{(m)}|^2\,dt$.
We let $\|\phi\|_{m,\Gamma_T}^2=\sum_{i=1}^n\|\phi_{{i}}\|_{m,\Gamma_{{T}}}^2$,
where $\|\phi_i\|_{m,\Gamma_T}^2=\sum_{k\leq m}
\int_{\Gamma_T}|\phi_i^{(k)}|^2\,dt$.
For $f(t),g(t)\in \mathbb{R}^n$ defined on $\Gamma_T$, we define the
inner products $\langle f,g\rangle=\sum_{i=1}^nf_ig_i$ and
$\langle f,g\rangle_{\Gamma_T}=\int_{\Gamma_T}\left(\sum_{i=1}^n
f_ig_i\right)\,dt$.

\section{A reformulation of $S_T$}\label{sec:reformulation_ST}
We start with a necessary condition
given by Maupertuis' principle of least action for the minimizer
$(T^*,\phi^*)$ of Problem II.
\begin{lemma}[\cite{Heymann_CPAM08}]
Let $(T^*,\phi^*)$ be the minimizer of Problem II. Then $\phi^*$ is located on
the surface $H(\phi,\frac{\partial L}{\partial\phi'})=0$, where $H$ is the
Hamiltonian given by the Legendre transform of $L(\phi,\phi'):=\tfrac{1}{2}
|\phi'-b(\phi)|^2$. More specifically,
for equation \eqref{eqn:sode}
\begin{equation}\label{eqn:H_constraint}
H(\phi,\frac{\partial L}{\partial\phi'})=0\quad\Longleftrightarrow\quad
|\phi'(t)|=|b(\phi(t))|,\quad\forall\, t.
\end{equation}
\end{lemma}

We will call equation \eqref{eqn:H_constraint} the zero-Hamiltonian
constraint in this paper.
The zero-Hamiltonian constraint defines a nonlinear mapping between the arc
length of the geometrically fixed lines on surface $H=0$ and time $t$
(see Section \ref{sec:escape} for more details). We instead consider a
linear time scaling on ${\Gamma_T}$, which is simpler and more flexible for numerical
approximation. For any given
transition path $\phi$ and a fixed $T$, we consider the change of variable
$s=t/T\in[0,1]=\Gamma_1$. Let $\phi(t)=\phi(sT)=:\bar{\phi}(s).$ Then
$\bar{\phi}'(s)=\phi'(t)T,$ and we rewrite the action functional as
\begin{equation}
S_T(\phi(t))=S_T(\bar{\phi}(s))=\frac{T}{2}\int_0^1\left|T^{-1}
\bar{\phi}'(s)-b(\bar{\phi}(s))\right|^2\,ds=:S(T,\bar{\phi}).
\end{equation}

\begin{lemma}\label{lem:linear_scaling}
For any given transition path $\phi$, we have
\begin{equation}
\hat{S}(\bar{\phi}):=S(\hat{T}(\bar{\phi}),\bar{\phi})
=\inf_{T\in\mathbb{R}^+}S(T,\bar{\phi}),
\end{equation}
if $\hat{T}(\bar{\phi})<\infty$, where
\begin{equation}\label{eqn:T_opt}
\hat{T}(\bar{\phi})=\frac{|\bar{\phi}'|_{0,\Gamma_1}}
{|b(\bar{\phi})|_{0,\Gamma_1}}.
\end{equation}
\end{lemma}
\begin{proof}
It is easy to verify that the functional $\hat{T}(\bar{\phi})$ is nothing but
the unique solution of the optimality condition $\partial_TS(T,\bar{\phi})=0$.
\end{proof}
\begin{corollary}\label{cor:T_opt}
Let $(T^*,\phi^*)$ be the minimizer of Problem II. If $T^*<\infty$, we have
$T^*=\hat{T}(\bar{\phi}^*)$, where $\bar{\phi}^*(s){:=}\phi^*(sT^*)$.
\end{corollary}
\begin{proof}
From the zero-Hamiltonian constraint \eqref{eqn:H_constraint} and the definition of $\bar{\phi}$, we have
\[
|(\bar{\phi}^*)'|=|(\phi^*)'|T^*=|b(\phi^*)|T^*=|b(\bar{\phi}^*)|T^*.
\]
Integrating both sides on $\Gamma_1$, we have the conclusion.
\end{proof}

For any absolutely continuous path $\phi$, it is shown in Theorem 5.6.3 in
\cite{Dembo98} that $S_T$ can be written as
\begin{equation}\label{eqn:S_H1}
S_T(\phi)=\left\{
\begin{array}{rl}
S_T(\phi),&\phi\in H^1(\Gamma_T;\mathbb{R}^n),\\
\infty,&\textrm{otherwise}.
\end{array}
\right.
\end{equation}
This means that we can seek the MAP in the Sobolev space
$H^1(\Gamma_T;\mathbb{R}^n)$. From now on, we will use $H^1(\Gamma_T)$
to indicate $H^1(\Gamma_T;\mathbb{R}^n)$ if no ambiguity arises. The same
rule will be applied to other spaces such as $H_0^1(\Gamma;\mathbb{R}^n)$
and $L^2(\Gamma;\mathbb{R}^n)$.

We define the following two admissible sets consisting of transition paths:
\begin{align}
\cA_T&=\big\{\phi\in H^1(\Gamma_T)\,{:}\;\phi(0)=0,\,\phi(T)=x\big\},\\
\cA_1&=\big\{\bar{\phi}\in H^1(\Gamma_1)\,{:}\;\bar{\phi}(0)=0,\,\bar{\phi}(1)=x\big\},
\end{align}
where we let $x_1=0$ and $x_2=x$ just for convenience.
\begin{lemma}\label{lem:opt_linear_scaling}
If $T^*<\infty$, we have
\begin{equation}\label{eqn:LDP-wo-T}
S_{T^*}(\phi^*)=\hat{S}(\bar{\phi}^*)=\inf_{\bar{\phi}\in\cA_1}\hat{S}
(\bar{\phi}).
\end{equation}
and $T^*=\hat{T}(\bar{\phi}^*)$ $($see equation \eqref{eqn:T_opt}$)$, where 
$\phi^*(t)=\bar{\phi}^*(t/T^*)$ $($or $\bar{\phi}^*(s)=\phi^*(sT^*))$.
\end{lemma}
\begin{proof}
If $\left(T^*,\phi^*\right)$ is a minimizer of $S_T(\phi)$ and $T^*<\infty,$
then
\[
S_{T^*}(\phi^*)=\inf_{T\in\mathbb{R}^+}\inf_{\phi\in\mathcal{A}_T}S_T(\phi)
=\inf_{\bar{\phi}\in\mathcal{A}_1}\hat{S}(\bar{\phi})\leq\hat{S}(\bar{\phi}^*),
\]
and
\[
\hat{S}(\bar{\phi}^*)=\inf_{T\in\mathbb{R}^+}S(T,\bar{\phi}^*)
=\inf_{T\in\mathbb{R}^+}S_T(\phi^*)\leq S_{T^*}(\phi^*).
\]
Thus, $S_{T^*}(\phi^*)=S(T^*,\bar{\phi}^*)=\hat{S}(\bar{\phi}^*),$ that is,
$\bar{\phi}^*$ is a minimizer of $\hat{S}(\bar{\phi})$ for
$\bar{\phi}\in\mathcal{A}_1,$ and $T^*=\hat{T}(\bar{\phi}^*)$ from Corollary
\ref{cor:T_opt}.

Conversely, if $\bar{\phi}^*$ is a minimizer of $\hat{S}(\bar{\phi})$, we let
$T^*=\hat{T}(\bar{\phi}^*),$ and $\phi^*(t)=\bar{\phi}^*(\frac{t}{T^*}),$ for
$t\in[0,T^*].$ We have
\[
S_{T^*}(\phi^*)=S(\hat{T}(\bar{\phi}^*),\bar{\phi}^*)=\hat{S}(\bar{\phi}^*)
=\inf_{\bar{\phi}\in\mathcal{A}_1}\hat{S}(\bar{\phi})=\inf_{T\in\mathbb{R}^+}
\inf_{\phi\in\mathcal{A}_T}S_T(\phi),
\]
when $T^*<\infty$. Then $(T^*, \phi^*)$ is a minimizer of $S_T(\phi)$. So the
minimizers of $\hat{S}(\bar{\phi})$ and $S_T(\phi)$ have a one-to-one
correspondence when the optimal integral time is finite.
\end{proof}

Lemma \ref{lem:opt_linear_scaling} shows that for a finite $T^*$ we
can use equation \eqref{eqn:LDP-wo-T} instead of Problem II to approximate
the quasi-potential such that the optimization parameter $T$ is removed {and we obtain a new problem
\begin{equation}\label{problem_reformulation}
\hat{S}(\bar{\phi}^*)=\inf_{\substack{\bar{\phi}(0)=x_1,\\
\bar{\phi}(1)=x_2}}\hat{S}(\bar{\phi})
\end{equation}
that is equivalent to Problem II}.

\section{Finite element discretization of Problems I and II}
The numerical method to approximate Problems I and II {is} usually called
minimum action method (MAM) \cite{E_CPAM04}. Many versions of MAM have
been developed, where the action functional is discretized by either
finite difference method or finite element method. In this work, we
consider the finite element discretization of $S_T(\phi)$ and focus on
the convergence of the finite element approximation of the minimizer.

Let $\mathcal{T}_h$ and
$\overline{\mathcal{T}}_h$ be {partitions} of $\Gamma_T$ and $\Gamma_1$,
respectively. We define the following approximation spaces given by linear
finite elements:
\begin{align*}
\mathcal{B}_h&=\big\{\phi_h\in \mathcal{A}_T{\,:\:}\phi_h|_I
\text{ is affine {for} each } I\in\mathcal{T}_h\big\},\\
\overline{\mathcal{B}}_h&=\big\{\bar{\phi}_h\in\cA_1{\,:\:}\bar{\phi}_h|_I
\textrm{ is affine {for} each }I\in
\overline{\mathcal{T}}_h\big\}.
\end{align*}
For any $h,$ we define the following discretized action functionals:
\begin{equation}
S_{T,h}(\phi_h)=\bigg\{\begin{array}{ll}\frac{1}{2}\int_0^T |\phi_h'-
b(\phi_h)|^2\,dt, & \text{if }\phi_h\in\mathcal{B}_h\\
\infty, & \text{if } \phi_h\not\in\mathcal{B}_h,
\end{array}
\end{equation} and
\begin{equation}
\hat{S}_{h}(\bar{\phi}_h)=\bigg\{\begin{array}{ll}
\frac{\hat{T}(\bar{\phi}_h)}{2}\int_0^1 |
\frac{1}{\hat{T}(\bar{\phi}_h)}\bar{\phi}_h'-
b(\bar{\phi}_h)|^2\,dt,& \text{if }
\bar{\phi}_h\in\overline{\mathcal{B}}_h,\\
\infty, & \text{if } \bar{\phi}_h\not\in\overline{\mathcal{B}}_h.
\end{array}
\end{equation}

We note that for a fixed integration time $T$, we can rewrite $S_T(\phi)$
as $\hat{S}(\bar{\phi})$ by letting $T=\hat{T}$, such that Problem I can
also be defined on $\Gamma_1$. Since we intend to use the reformulation
$\hat{S}(\bar{\phi})$ to deal with the parameter $T$ in Problem II,
we use $\Gamma_T$ and $\Gamma_1$ to define Problem I and II, respectively,
for clarity.

\subsection{Problem I with a fixed $T$}
For this case, our main results are summarized in the following theorem:
\begin{theorem}\label{thm:cong_minima_ST}
For Problem I with a fixed $T$, we have
\[
\min_{\phi\in\mathcal{A}_T}S_T(\phi)=\lim_{h\rightarrow0}
\inf_{\phi_h\in\mathcal{B}_h}S_{T,h}(\phi_h),
\]
namely, the minima of $S_{T,h}$ converge to the minimum of $S_T(\phi)$
as $h\rightarrow0.$ Moreover, if $\{\phi_h\}\subset\mathcal{B}_h$ is a
sequence of minimizers of $S_{T,h},$ then there is a subsequence that
converges weakly in $H^1(\Gamma_T)$ to some $\phi\in\mathcal{A}_T,$
which is a minimizer of $S_T.$
\end{theorem}

The proof of this theorem will be split into two steps: (1) the existence
of the minimizer of $S_T(\phi)$ in $\mathcal{A}_T$, and (2)
$\Gamma$-convergence of $S_{T,h}$ to $S_T$ as $h\rightarrow0$.
\subsubsection{Solution existence in $\mathcal{A}_T$}
We search the minimizer of $S_T(\phi)$ in the admissible set $\cA_T$.
The solution existence is given by the following lemma.
\begin{lemma}\label{lem:sln_existence_fixed_T}
There exists at least one function $\phi^*\in\mathcal{A}_T$ such that
\[
S_T(\phi^*)=\min_{\phi\in\mathcal{A}_T}S_T(\phi).
\]
\end{lemma}
\begin{proof}
We first establish the coerciveness of $S_T(\phi)=\frac{1}{2}\int_0^T|
\phi'-b(\phi)|^2\,dt.$ In order to do so, we define an auxiliary function
$g$ by
\[
g(t)=\phi(t)-\int_0^t b(\phi(u))\,du.
\]
Then $g'=\phi'-b(\phi)$ and $g(0)=0.$ Since $b(x)$ is globally
Lipschitz continuous, we have
\begin{align*}
|\phi'(t)|&\leq |b(\phi(t))-b(0)|+|b(0)|+|g'|\\
&\leq K|\phi|+|b(0)|+|g'(t)|\\
&\leq K\int_0^t |\phi'(s)|\,ds+|b(0)|+|g'(t)|.
\end{align*}
By Gronwall's inequality, we have
\[
|\phi'(t)|\leq K\int_0^t(|b(0)|+|g'(s)|)e^{K(t-s)}\,ds+|b(0)|+|g'(t)|,
\]
from which we obtain
\[
|\phi|_{1,\Gamma_T}\leq C_1|b(0)|^2+C_2|g|_{1,\Gamma_T}^2,
\]
where $C_1$ and $C_2$ are two positive constants depending on $K$
and $T$.
Thus, the action functional satisfies
\[
S_T(\phi)=\frac{1}{2}|g|_{1,\Gamma_T}^2\geq \frac{1}{2}C_2^{-1}
|\phi|_{1,\Gamma_T}^2-\frac{1}{2}C_1C^{-1}_2|b(0)|^2.
\]
The coerciveness follows. On the other hand, the integrand
$|\phi'-b(\phi)|^2$ is bounded below by $0,$ and convex in $\phi'.$
By the Theorem 2 on Page 448 in \cite{Evans10}, $S_T(\phi)$ is weakly
lower semicontinuous on $H^1(\Gamma_T).$

For any minimizing sequence $\{\phi_k\}_{k=1}^{\infty},$ from the
coerciveness, we have
\[
\sup_k |\phi_k|_{1,\Gamma_T}<\infty.
\]
Let $\phi_0\in\mathcal{A}_T$ be any fixed function, e.g.{,} the linear
function on $\Gamma_T$ from $0$ to $x.$ Then $\phi_k-\phi_0\in
H_0^1(\Gamma_T),$ and
\begin{align*}
|\phi_k|_{0,\Gamma_T}&\leq|\phi_k-\phi_0|_{0,\Gamma_T}
+|\phi_0|_{0,\Gamma_T}\\
&\leq C_p|\phi_k-\phi_0|_{1,\Gamma_T}+|\phi_0|_
{0,\Gamma_T}<\infty,
\end{align*}
by the Poincar\'{e}'s Inequality. Thus $\{\phi_k\}_{k=1}^{\infty}$ is
bounded in $H^1(\Gamma_T).$ Then there exists a subsequence
$\{\phi_{k_j}\}_{j=1}^{\infty}$ converging weakly to some
$\phi^*\in H^1(\Gamma_T)$ in $H^1(\Gamma_T).$ Then
$\phi_{k_j}-\phi_0$ converges to $\phi^*-\phi_0$ weakly in
$H_0^1(\Gamma_T).$ By Mazur's Theorem \cite{Evans10},
$H_0^1(\Gamma_T)$ is weakly closed. So
$\phi^*-\phi_0\in H_0^1(\Gamma_T)$, i.e., $\phi^{*}\in\mathcal{A}_T.$

Therefore,
$S_T(\phi^*)\leq\liminf_{j\rightarrow\infty} S_T(\phi_{k_j})
=\inf_{\phi\in\mathcal{A}_T}S_T(\phi).$ Since
$\phi^{*} \in\mathcal{A}_T,$ we reach the conclusion.
\end{proof}

\subsubsection{$\Gamma$-convergence of $S_{T,h}$}
We first note the following simple property:
\begin{property}\label{prop:for_gamma}
For any sequence $\{\phi_h\}\subset\mathcal{B}_h$ converging weakly to
$\phi\in H^1(\Gamma_T)$, we have
\[
\lim_{h\rightarrow0}|b(\phi_h)-b(\phi)|_{0,\Gamma_T}=0.
\]
\end{property}
\begin{proof}
Since $\phi_h$ converges weakly to $\phi$ in $H^1(\Gamma_T)$,
$\phi_h\rightarrow\phi$ in $L^2(\Gamma_T)$, i.e., $\phi_h$ converges
strongly to $\phi$ in the $L^2$ sense. By the Lipschitz continuity of 
$b$, we reach the conclusion.
\end{proof}

We now establish the $\Gamma$-convergence of $S_{T,h}$:
\begin{lemma}[$\Gamma$-convergence of $S_{T,h}$]
\label{lem:GammaConvergence_ST}
Let $\{\mathcal{T}_h\}$ be a sequence of finite element meshes with
$h\rightarrow0$. For every $\phi\in\mathcal{A}_T$, the following two
properties hold:
\begin{itemize}
\item Lim-inf inequality: for every sequence $\{\phi_h\}$ converging
weakly to $\phi$ in $H^1(\Gamma_T),$ we have
\begin{equation}\label{GammaConvergence1_ST}
S_T(\phi)\leq\liminf_{h\rightarrow0}S_{T,h}(\phi_h).
\end{equation}
\item Lim-sup inequality: there exists a sequence
$\{\phi_h\}\subset\mathcal{B}_h$ converging weakly to $\phi$ in
$H^1(\Gamma_T),$ such that
\begin{equation}\label{GammaConvergence2_ST}
S_T(\phi)\geq\limsup_{h\rightarrow0}S_{T,h}(\phi_h).
\end{equation}
\end{itemize}
\end{lemma}

\begin{proof}
We first address the lim-inf inequality. We only need to consider
a sequence $\{\phi_h\}\subset\mathcal{B}_h,$ since otherwise,
\eqref{GammaConvergence1_ST} is trivial by the definition of
$S_{T,h}(\phi).$ Let $\{\phi_h\}\subset\mathcal{B}_h$ be an arbitrary 
sequence converging weakly to $\phi$ in $H^1(\Gamma_T)$. The action functional
can be written as
\begin{align}
&\int_0^T |\phi_h'-b(\phi_h)|^2\,dt\nonumber\\
&=\int_0^T |\phi_h'|^2\,dt+\int_0^T |b(\phi_h)|^2\,dt-
2\int_0^T \langle\phi_h',b(\phi_h)\rangle\,dt=I_1+I_2+I_3.
\label{thm3-2}
\end{align}

The functional defined by $I_1$ is obviously weakly lower
semicontinuous in $H^1(\Gamma_T)$ since the integrand is convex
with respect to $\phi'$.

For $I_2$ in equation \eqref{thm3-2}. Using Property
\ref{prop:for_gamma}, we have
\begin{align*}
\lim_{h\rightarrow0}|b(\phi_h)|_{0,\Gamma_T}
=|b(\phi)|_{0,\Gamma_T},
\end{align*}

For $I_3$ in equation \eqref{thm3-2}. We have
\begin{align*}
&\left|\int_0^T \langle\phi_h',b(\phi_h)\rangle\,dt
-\int_0^T \langle\phi',b(\phi)\rangle\,dt\right|\\
=&\left|\int_0^T \langle\phi_h',b(\phi_h)-b(\phi)
\rangle\,dt+\int_0^T \langle\phi_h'-\phi',b(\phi)\rangle\,dt\right|\\
\leq&|\phi_h|_{1,\Gamma_T}|b(\phi_h)
-b(\phi)|_{0,\Gamma_T}+|\langle\phi_h'-\phi',b(\phi)\rangle_
{\Gamma_T}|.
\end{align*}
Using Property \ref{prop:for_gamma} and the fact that
$\sup_h|\phi_h|_{1,\Gamma_T}<\infty$, we have {that} the first term
of the above inequality converges to $0$. Moreover, the second term
also converges to $0$ due to the weak convergence of $\phi_h$
to $\phi$ {in $H^1(\Gamma_T)$}. Thus,
\[
\lim_{h\rightarrow0}\int_0^T \langle\phi_h',b(\phi_h)
\rangle\,dt=\int_0^T \langle\phi',b(\phi)\rangle\,dt.
\]
Combining the results for $I_1$, $I_2$ and $I_3$, we obtain
\begin{align*}
&\liminf_{h\rightarrow0}\int_0^T |\phi_h'
-b(\phi_h)|^2\,dt\\
=&\liminf_{h\rightarrow0}\left[\int_0^T |\phi_h'|^2\,dt+\int_0^T |
b(\phi_h)|^2\,dt-2\int_0^T
\langle\phi_h',
b(\phi_h)\rangle\,dt\right]\\
=&\liminf_{h\rightarrow0}\int_0^T |\phi_h'|^2\,dt
+\lim_{h\rightarrow0}\int_0^T |b(\phi_h)|^2\,dt
-2\lim_{h\rightarrow0}\int_0^T
\langle\phi_h',b(\phi_h)\rangle\,dt\\
\geq&\int_0^T |\phi'|^2\,dt+\int_0^T |b(\phi)|^2\,dt
-2\int_0^T \langle\phi',b(\phi)\rangle\,dt\\
=&\int_0^T |\phi'-b(\phi)|^2\,dt,
\end{align*}
which yields the lim-inf inequality.

We now address the lim-sup inequality. Since $H^2(\Gamma_T)$ is
{dense} in $H^1(\Gamma_T)$, for any
$\phi\in H^1(\Gamma_T),$ and $\varepsilon>0$, there exists a
non-zero $u_\varepsilon\in H^2(\Gamma_T),$ such that
$\|\phi-u_\varepsilon\|_{{1},\Gamma_T}<\varepsilon.$ We have
\[
|\mathcal{I}_h u_\varepsilon-u_\varepsilon|_{1,\Gamma_T}
\leq ch|u_\varepsilon|_{2,\Gamma_T}\leq c\varepsilon,
\]
by letting
\[
h=h(\varepsilon)=\min\{\frac{\varepsilon}
{|u_\varepsilon|_{1,\Gamma_T}},\frac{\varepsilon}
{|u_\varepsilon|_{2,\Gamma_T}},\varepsilon\},
\]
{where $\mathcal{I}_h$ is an interpolation operator defined by linear finite 
	elements.} 
Let $\phi_h=\mathcal{I}_h u_\varepsilon.$ Then we have
$\phi_h\in\mathcal{B}_h,$ and
\begin{align*}
|\phi_h-\phi|_{1,\Gamma_T}
\leq&|\phi_h-u_\varepsilon|_{1,\Gamma_T}
+|u_\varepsilon-\phi|_{1,\Gamma_T}\\
=&|\mathcal{I}_h u_\varepsilon-u_\varepsilon|_{1,\Gamma_T}
+|u_\varepsilon-\phi|_{1,\Gamma_T}\\
<&c\varepsilon+\varepsilon\rightarrow0,
\end{align*}
and
\begin{align*}
|\phi_h-\phi|_{0,\Gamma_T}
\leq&|\phi_h-u_\varepsilon|_{0,\Gamma_T}+|u_\varepsilon
-\phi|_{0,\Gamma_T}\\
=&|\mathcal{I}_h u_\varepsilon-u_\varepsilon|_{0,\Gamma_T}
+|u_\varepsilon-\phi|_{0,\Gamma_T}\\
\leq& ch|u_\varepsilon|_{1,\Gamma_T}+\varepsilon\\
<&c\varepsilon+\varepsilon\rightarrow0,
\end{align*}
as $\varepsilon\rightarrow0.$ So $\phi_h$ converges to $\phi$ in
$H^1(\Gamma_T),$ and also converges weakly in $H^1(\Gamma_T)$.
By Property \ref{prop:for_gamma}, we know that
$b(\phi_h)\rightarrow b(\phi)$ in $L_2(\Gamma_T)$.
Thus,
\[
\lim_{h\rightarrow0}S_{T,h}(\phi_h)=\lim_{h\rightarrow0}\frac{1}{2}
|\phi_h'-b(\phi_h)|_{0,\Gamma_T}^2=S_T(\phi),
\]
which yields the lim-sup equality.
\end{proof}

\subsubsection{Proof of Theorem \ref{thm:cong_minima_ST}}

{With the solution existence and the $\Gamma$-convergence being 
	proved, we only need the equi-coerciveness of $S_{T,h}$ for 
	the final conclusion. For any $\phi_h\in\mathcal{B}_h$, we have
	$S_{T,h}(\phi_h)=S_T(\phi_h).$ Then the equi-coerciveness of 
		$S_{T,h}$ in $\mathcal{B}_h$ follows from the coerciveness of 
		$S_T(\phi_h)$ restricted to $\mathcal{B}_h\subset\mathcal{A}_T$ 
		(see the first step in the proof of Lemma \ref{lem:sln_existence_fixed_T}).}

\subsection{Problem II with a finite $T^*$}
For this case, we consider the reformulation of $S_T$ given in
Section \ref{sec:reformulation_ST}. From Lemma
\ref{lem:opt_linear_scaling}, we know that Problem II with a finite
$T^*$ is equivalent to minimizing $\hat{S}$ in $\cA_1$ (see equation
\eqref{eqn:LDP-wo-T}). Our main results are summarized in the
following theorem:
\begin{theorem}\label{thm:cong_minima_Shat}
For Problem II with a finite $T^*$, we have
\[
\min_{\bar{\phi}\in\mathcal{A}_1}\hat{S}(\bar{\phi})
=\lim_{h\rightarrow0}
\inf_{\bar{\phi}_h\in\overline{\mathcal{B}}_h}\hat{S}_h(\bar{\phi}_h),
\]
namely, the minima of $\hat{S}_h$ converge to the minimum of
$\hat{S}$ as $h\rightarrow0.$ Moreover, if
$\{\bar{\phi}_h\}\subset\overline{\mathcal{B}}_h$ is a sequence of
minimizers of $\hat{S}_h,$ then there is a subsequence that
converges weakly in $H^1(\Gamma_1)$ to some
$\bar{\phi}\in\mathcal{A}_1,$ which is a minimizer of $\hat{S}.$
\end{theorem}

Similar to Problem I with a fixed $T$, we split the proof of this
theorem into two steps: (1) the existence of the minimizer of
$\hat{S}(\bar{\phi})$ in $\mathcal{A}_1$, and (2)
$\Gamma$-convergence of $\hat{S}_{h}$ to $\hat{S}$ as $h\rightarrow0$.

\subsubsection{Solution existence in $\mathcal{A}_1$}
We start from the following property of the functional $\hat{T}$.

\begin{property}\label{prop:T_bound}
There exists a constant $C_{\hat{T}}>0$ such that
\begin{equation}
\hat{T}(\bar{\phi})\geq C_{\hat{T}}
\end{equation}
for any $\bar{\phi}\in\mathcal{A}_1$.
\end{property}
\begin{proof}
For any $\bar{\phi}\in\mathcal{A}_1,$ let
$\bar{\phi}=\bar{\phi}_0+\bar{\phi}_L$, where
$\bar{\phi}_0\in H_0^1(\Gamma_1)$ and
$\bar{\phi}_L(s)=xs, s\in[0,1]$ {is} a linear function connecting $0$
and $x$. We have
\begin{align*}
\hat{T}(\bar{\phi})&=\frac{|\bar{\phi}_0'+x|_{0,\Gamma_1}}
{|b(\bar{\phi}_0+\bar{\phi}_L)|_{0,\Gamma_1}}\\
&\geq \frac{|\bar{\phi}_0'+x|_{0,\Gamma_1}}
{|b(\bar{\phi}_0+\bar{\phi}_L)-b(\bar{\phi}_L)|_{0,\Gamma_1}+
|b(\bar{\phi}_L)|_{0,\Gamma_1}}\\
&\geq \frac{|\bar{\phi}_0'+x|_{0,\Gamma_1}}{K|\bar{\phi}_0|
_{0,\Gamma_1}+|b(\bar{\phi}_L)|_{0,\Gamma_1}}\\
&\geq \frac{|\bar{\phi}_0'+x|_{0,\Gamma_1}}
{KC_p |\bar{\phi}_0'|_{0,\Gamma_1}+|b(\bar{\phi}_L)|_
{0,\Gamma_1}},
\end{align*}
where $C_p$ is the constant {for} Poincar\'{e}{'s} Inequality. So
\begin{align*}
\hat{T}(\bar{\phi})^2&\geq \frac{|\bar{\phi}_0'+
x|_{0,\Gamma_1}^2}{2K^2C_p^2
|\bar{\phi}_0'|
_{0,\Gamma_1}^2+2|b(\bar{\phi}_L)|_{0,\Gamma_1}^2}\\
&=\frac{|\bar{\phi}_0'+x|_{0,\Gamma_1}^2}{C_1 |\bar{\phi}_0'|
_{0,\Gamma_1}^2+C_2}\\
&=:J(\bar{\phi}_0)>0,
\end{align*}
where $C_1=2K^2C_p^2>0,$ and $C_2=2|b(\bar{\phi}_L)|_
{0,\Gamma_1}^2>0.$

Let $\delta\bar{\phi}\in H_0^1(\Gamma_1)$ be a perturbation function
with $\delta\bar{\phi}(0)=\delta\bar{\phi}(1)=0$. We have
\begin{align*}
&J(\bar{\phi}_0+\delta\bar{\phi})-J(\bar{\phi}_0)\\
=&\frac{|\bar{\phi}_0'+x+\delta\bar{\phi}'|^2
_{0,\Gamma_1}}{C_1|\bar{\phi}_0'+\delta\bar{\phi}'|^2
_{0,\Gamma_1}+C_2}-
\frac{|\bar{\phi}_0'+x|^2
_{0,\Gamma_1}}{C_1|\bar{\phi}_0'|^2_{0,\Gamma_1}+C_2}\\
=&\frac{|\bar{\phi}_0'+x+\delta\bar{\phi}'|^2
_{0,\Gamma_1}(C_1|\bar{\phi}_0'|^2_{0,\Gamma_1}+C_2)-
|\bar{\phi}_0'+x|^2
_{0,\Gamma_1}(C_1|\bar{\phi}_0'+\delta\bar{\phi}'|^2
_{0,\Gamma_1}+C_2)}{(C_1|\bar{\phi}_0'+\delta\bar{\phi}'|^2
_{0,\Gamma_1}+C_2)(C_1|\bar{\phi}_0'|^2_{0,\Gamma_1}+C_2)}\\
=&\frac{2\langle\bar{\phi}_0'+x,\delta\bar{\phi}'\rangle_{\Gamma_1}(C_1
|\bar{\phi}_0'|^2_{0,\Gamma_1}+C_2)-2C_1\langle\bar{\phi}_0',
\delta\bar{\phi}'\rangle_{\Gamma_1}|\bar{\phi}_0'+x|^2
_{0,\Gamma_1}}{(C_1|\bar{\phi}_0'|^2
_{0,\Gamma_1}+C_2)^2}+R(\bar{\phi}_0',x,\delta\bar{\phi}'),
\end{align*}
where $R$ is the remainder term of $O(|\delta\bar{\phi}|^2
_{1,\Gamma_1})$.

We then have the first-order variation of $J$ as
\[
\delta J=\frac{2\langle\bar{\phi}_0',\delta\bar{\phi}'\rangle_{\Gamma_1}
(C_1|\bar{\phi}_0'|^2_{0,\Gamma_1}+C_2-C_1|\bar{\phi}_0'+x|^2
_{0,\Gamma_1})}{(C_1|\bar{\phi}_0'|^2
_{0,\Gamma_1}+C_2)^2}.
\]
The optimality condition $\delta J=0$ yields two possible cases:
$\bar{\phi}_0'=0$ and $C_1|\bar{\phi}_0'|^2_{0,\Gamma_1}+C_2
=C_1|\bar{\phi}_0'+x|^2_{0,\Gamma_1}.$ For the first case,
$\bar{\phi}_0$ is a constant. But $\bar{\phi}_0\in H_0^1(\Gamma_1),$
so $\bar{\phi}_0=0.$ Then $J(0)=\frac{|x|^2}{C_2}>0.$ For the
second case, $J(\bar{\phi}_0)=\frac{1}{C_1}>0.$ Thus,
\[
\hat{T}^2(\bar{\phi})\geq \min\{\frac{|x|^2}{C_2},\frac{1}{C_1}\}.
\]
More specifically,
\[
\hat{T}(\bar{\phi})\geq C_{\hat{T}}:= \min\{\frac{|x|}{\sqrt{2}
|b(\bar{\phi}_L)|^2_{0,\Gamma_1}},\frac{1}{\sqrt{2}KC_p}\}>0.
\]
\end{proof}

We search the minimizer of $\hat{S}(\bar{\phi})$ in the admissible
set $\cA_1$. The solution existence is given by the following lemma.

\begin{lemma}\label{lem:quasi_potential_finite_T_solution_existence}
If the optimal integral time $T^*$ for Problem II is finite, there
exists at least one function $\phi^*\in\mathcal{A}_T$ such that
\[
S_{T^*}(\phi^*)=\min_{\substack{T\in\mathbb{R}^+,\\
\phi\in\mathcal{A}_T}}S_T(\phi)=\min_{\bar{\phi}\in\mathcal{A}_1}
\hat{S}(\bar{\phi}).
\]
\end{lemma}
\begin{proof}
We first establish the weakly lower semi-continuity of
$\hat{S}(\bar{\phi})$ in $H^1(\Gamma_1)$. Rewrite
$\hat{S}(\bar{\phi})$ by substituting \eqref{eqn:T_opt} to get
\begin{align*}
\hat{S}(\bar{\phi})&=\frac{\hat{T}(\bar{\phi})}{2}\int_0^1 \left|
\hat{T}^{-1}(\bar{\phi})\bar{\phi}'-b(\bar{\phi})\right|^2\,dt\\
&=|\bar{\phi}'|_{0,\Gamma_1}|b(\bar{\phi})|_{0,\Gamma_1}
-\langle\bar{\phi}',b(\bar{\phi})\rangle_{\Gamma_1}.
\end{align*}
For any sequence $\bar{\phi}_k$ {converging} weakly to $\bar{\phi}$
in $H^1(\Gamma_1)$, $\{\bar{\phi}_k'\}$ is bounded in
$L^2(\Gamma_1)$ and $\bar{\phi}_k\rightarrow\bar{\phi}$ in
$L^2(\Gamma_1)$. Coupling with the global Lipschitz continuity of $b$,
we can obtain
\begin{align*}
\lim_{k\rightarrow\infty}|b(\bar{\phi}_k)|_{0,\Gamma_1}^2&
=|b(\bar{\phi})|_{0,\Gamma_1}^2\\
\lim_{k\rightarrow\infty}\langle\bar{\phi}_k',b(\bar{\phi}_k)
\rangle_{\Gamma_1}&=\langle\bar{\phi}',b(\bar{\phi})
\rangle_{\Gamma_1}.
\end{align*}
The weakly lower semicontinuity of
$|\bar{\phi}'|_{0,\Gamma_1}$ yields that
\begin{equation}\label{thm:finite_T_3}
\liminf_{k\rightarrow\infty}|\bar{\phi}_k'|_{0,\Gamma_1}
\geq|\bar{\phi}'|_{0,\Gamma_1}.
\end{equation}
Combining the above results, we obtain
\begin{align*}
&\liminf_{k\rightarrow\infty}\hat{S}_{k}(\bar{\phi}_k)\\
=&\liminf_{k\rightarrow\infty}\left(|\bar{\phi}_k'|_
{0,\Gamma_1}|b(\bar{\phi}_k)|_{0,\Gamma_1}-
\langle\bar{\phi}_k',b(\bar{\phi}_k)\rangle_{\Gamma_1}\right)\\
=&\liminf_{k\rightarrow\infty}|\bar{\phi}_k'|_{0,\Gamma_1}
|b(\bar{\phi}_k)|_{0,\Gamma_1}-\lim_{k\rightarrow\infty}
\langle\bar{\phi}_k',b(\bar{\phi}_k)\rangle_{\Gamma_1}\\
\geq&|\bar{\phi}'|_{0,\Gamma_1}|b(\bar{\phi})|
_{0,\Gamma_1}-\langle\bar{\phi}',b(\bar{\phi})
\rangle_{\Gamma_1}\\
=&\hat{S}(\bar{\phi}),
\end{align*}
that is, $\hat{S}(\bar{\phi})$ is weakly lower semicontinuous
in $H^1(\Gamma_1).$

We subsequently establish the coercivity of $\hat{S}(\bar{\phi})$.
Since $T^*$ is finite, there exists $M\in(T^*,\infty),$ such that
\[
\inf_{\bar{\phi}\in\mathcal{A}_1}\hat{S}(\bar{\phi})=
\inf_{\substack{\bar{\phi}\in \mathcal{A}_1,\\
\hat{T}(\bar{\phi})<M}}\hat{S}(\bar{\phi}).
\]
In fact, {by} Lemma \ref{lem:opt_linear_scaling}, a minimizing
sequence $\{\bar{\phi}_k\}$ of $\hat{S}(\bar{\phi})$ defines a
minimizing sequence $\{(\hat{T}(\bar{\phi}_k),\bar{\phi}_k)\}$
of $S(T,\bar{\phi}),$ which also corresponds to a minimizing
sequence of $S_T(\phi)$.  {The assumption of $T^*<\infty$ 
allows us to add the condition that $\sup_k\hat{T}(\bar{\phi}_k)<M$. 
Otherwise, $\hat{T}(\bar{\phi}_k)$ must go to infinity. The continuity of 
$S(T,\bar{\phi})$ with respect to $T$ yields that $T^*=\infty$, which contradicts 
our assumption that $T^*<\infty$.}  
Now, let $\hat{T}^{-1}(\bar{\phi}){\bar{\phi}'(s)}-b(\bar{\phi}(s))
=\bar{g}'(s).$ Then for any $\bar{\phi}\in\mathcal{A}_1$ with
$\hat{T}(\bar{\phi})<M,$
\begin{align*}
|\bar{\phi}'|&\leq|\hat{T}(\bar{\phi})||b(\bar{\phi})|
+|\hat{T}(\bar{\phi})||g'|\\
&\leq M|b(\bar{\phi})|+M|\bar{g}'|\\
&\leq MK|\bar{\phi}|+M|b(0)|+M|\bar{g}'|\\
&\leq MK\int_0^s|\bar{\phi}'(u)|\,du+M|b(0)|+M|\bar{g}'|.
\end{align*}
By Gronwall's Inequality, we have
\[
|\bar{\phi}'(s)|\leq \int_0^s M^2K(|\bar{g}'(u)|+|b(0)|)
e^{KM(s-u)}\,du+M|b(0)|
+M|\bar{g}'(s)|,
\]
which yields that
\begin{equation}\label{eqn:phi_g_bound}
|\bar{\phi}'|^2_{0,\Gamma_1}\leq C_1|b(0)|^2
+C_2|\bar{g}'|^2_{0,\Gamma_1},
\end{equation}
where $C_1,C_2\in(0,\infty)$ only depend on $M$ and $K.$ So
\begin{align*}
\hat{S}(\bar{\phi})&=\frac{\hat{T}(\bar{\phi})}{2}\int_0^1
\left|\hat{T}^{-1}
(\bar{\phi})\bar{\phi}'(s)-b(\bar{\phi}(s))\right|^2\,ds\\
&=\frac{\hat{T}(\bar{\phi})}{2}|\bar{g}'|^2_{0,\Gamma_1}\\
&\geq\frac{C_{\hat{T}}}{2}\left(\frac{1}{C_2}|\bar{\phi}'|^2
_{0,\Gamma_1}-\frac{C_1}{C_2}|b(0)|^2\right),
\end{align*}
where we used Property \ref{prop:T_bound} in the last step.
This is the coercivity.

For any minimizing sequence
$\{\bar{\phi}_k\}_{k=1}^{\infty}$ of $\hat{S}(\bar{\phi})$,
we have
\[
\sup_k|\bar{\phi}_k'|_{0,\Gamma_1}\leq\frac{2C_1}{C_T}|b(0)|^2
+\frac{2C_2}{C_T}\sup_k\{\hat{S}(\bar{\phi}_k)\}<\infty.
\]
Let $\bar{\phi}_0\in\mathcal{A}_1$. Then
\begin{align*}
|\bar{\phi}_k|_{0,\Gamma_1}&\leq|\bar{\phi}_k
-\bar{\phi}_0|_{0,\Gamma_1}
+|\bar{\phi}_0|_{0,\Gamma_1}\\
&\leq C_p|\bar{\phi}_k'-\bar{\phi}_0'|_{0,\Gamma_1}
+|\bar{\phi}_0|_{0,\Gamma_1}<\infty
\end{align*}
by Poincar\'{e}'s Inequality. Thus, $\{\bar{\phi}_k\}_{k=1}^{\infty}$
is bounded in $H^1(\Gamma_1).$ Then there is a subsequence
$\{\bar{\phi}_{k_j}\}_{j=1}^{\infty}$ converging to some
$\bar{\phi}^*\in H^1(\Gamma_1)$ weakly in $H^1(\Gamma_1).$ So
$\bar{\phi}_{k_j}-\bar{\phi}_0$ converges weakly to
$\bar{\phi}^*-\bar{\phi}_0$ {in} $H^1_0(\Gamma_1).$ By Mazur's
Theorem, $H^1_0(\Gamma_1)$ is weakly closed. So
$\bar{\phi}^*-\bar{\phi}_0\in H^1_0(\Gamma_1),$ and
$\bar{\phi}^*\in\mathcal{A}_1.$
By Lemma \ref{lem:opt_linear_scaling}, $\phi^*\in\mathcal{A}_T$
corresponding to $\bar{\phi}^*\in\mathcal{A}_1$ is a minimizer
of $S_T(\phi)$ {and $T^*=\hat{T}(\bar{\phi}^*)$}.
\end{proof}

\subsubsection{$\Gamma$-convergence of $\hat{S}_h$}
The $\Gamma$-convergence of $\hat{S}_h$ with respect to parameter
$h$ {is} established in the following lemma:
\begin{lemma}[$\Gamma$-convergence of $\hat{S}_h$]
\label{lem:GammaConvergence_Shat}
Let $\{\mathcal{T}_h\}$ be a sequence of finite element meshes.
For every $\bar{\phi}\in\mathcal{A}_1$, the following two properties
hold:
\begin{itemize}
\item Lim-inf inequality: for every sequence $\{\bar{\phi}_h\}$
converging weakly to $\bar{\phi}$ in $H^1(\Gamma_1),$ we have
\begin{equation}\label{GammaConvergence1_Shat}
\hat{S}(\bar{\phi})\leq\liminf_{h\rightarrow0}\hat{S}_{h}
(\bar{\phi}_h).
\end{equation}
\item Lim-sup inequality: there exists a sequence
$\{\bar{\phi}_h\}\subset\hat{\mathcal{B}}_h$ converging weakly
to $\bar{\phi}$ in $H^1(\Gamma_1),$ such that
\begin{equation}\label{GammaConvergence2_Shat}
\hat{S}(\bar{\phi})\geq\limsup_{h\rightarrow0}\hat{S}_{h}
(\bar{\phi}_h).
\end{equation}
\end{itemize}
\end{lemma}

\begin{proof}
We first address the lim-inf inequality. We only consider
sequence $\{\bar{\phi}_h\}\subset\mathcal{B}_h,$ otherwise,
the inequality is trivial. Similar to the proof of Lemma
\ref{lem:quasi_potential_finite_T_solution_existence}, rewrite
the discretized functional as 
\begin{align*}
\hat{S}_{h}(\bar{\phi}_h)&=\frac{\hat{T}(\bar{\phi}_h)}{2}
\int_0^1 \left|\hat{T}^{-1}(\bar{\phi}_h)\bar{\phi}_h'
-b(\bar{\phi}_h)\right|^2\,dt\\
&=|\bar{\phi}_h'|_{0,\Gamma_1}|
b(\bar{\phi}_h)|_{0,\Gamma_1}-\langle\bar{\phi}_h',
b(\bar{\phi}_h)\rangle_{\Gamma_1}.
\end{align*}
By the same argument as in the proof of Lemma
\ref{lem:GammaConvergence_ST}, we have
\begin{align*}
\liminf_{h\rightarrow0}|\bar{\phi}_h'|_{0,\Gamma_1}
&\geq|\bar{\phi}'|_{0,\Gamma_1},\\
\lim_{h\rightarrow0}|b(\bar{\phi}_h)|_
{0,\Gamma_1}
&=|b(\bar{\phi})|_{0,\Gamma_1},\\
\lim_{h\rightarrow0}\langle\bar{\phi}_h',
b(\bar{\phi}_h)\rangle_{\Gamma_1}
&=\langle\bar{\phi}',b(\bar{\phi})\rangle_{\Gamma_1}.
\end{align*}
Combining these results, we have the lim-inf inequality.
The lim-sup inequality can be obtained by the same argument
as in the proof of Lemma \ref{lem:GammaConvergence_ST}
\end{proof}

\subsubsection{Proof of Theorem \ref{thm:cong_minima_Shat}}
{Similar to the proof of Theorem \ref{thm:cong_minima_ST}, the only thing 
	left is the verification of equi-coerciveness of $\hat{S}_{h}(\bar{\phi}_h)$, 
	which can be obtained directly from the coerciveness of $\hat{S}(\bar{\phi})$ 
    restricted onto $\bar{\mathcal{B}}_h\subset\mathcal{A}_1$ 
    (see the second step in the proof of Lemma \ref{lem:quasi_potential_finite_T_solution_existence}).}

\subsection{Problem II with an infinite $T^*$}
When $T^*$ is infinite, the integration domain becomes the whole real
space, corresponding to a degenerate case of linear scaling. To remove
the optimization parameter $T$, the zero-Hamiltonian constraint
\eqref{eqn:H_constraint} can be considered
under another assumption that the total arc length of $\phi^*$ is
finite, which is the basic idea of the geometric MAM (gMAM)
\cite{Heymann_CPAM08}. However, since the {Jacobian}
 {of the transform between time and arc length variables
 } will become singular
at critical points, the numerical accuracy will deteriorate when unknown
critical points exist along the MAP.

We will still work with the formulation with respect to time, which means
that we need to use a large but finite integration time to deal with
the {case} $T^*=\infty$. We discuss this case by considering a
relatively simple scenario, but the numerical difficulties are reserved.
Let $0\in D$ be an asymptotically stable equilibrium point, {$D$ is contained in the basin of attraction of $0$,} and
$\langle b(y),n(y)\rangle<0$ for any $y\in\partial D$, where $n(y)$ is the
exterior normal to the boundary $\partial D$. Then starting from any point in
$D$, a trajectory of system \eqref{eqn:dynamical_system} will
converge to $0$. We assume that the ending point $x$ of Problem II is
located on $\partial D$.

\subsubsection{Escape from the equilibrium point}\label{sec:escape}
If we consider the change of variable in general, say $\alpha=\alpha(t)$, we
have (see Lemma 3.1, Chapter 4 in \cite{FW})
\begin{equation}
S_T(\phi)\geq S(\tilde{\phi})=\int_{\alpha(0)}^{\alpha(T)}
(|\tilde{\phi}'||b|-\langle\tilde{\phi}',b\rangle)d\alpha,
\end{equation}
where $\tilde{\phi}(\alpha)=\phi(t(\alpha))$, $\tilde{\phi}'$ is the
derivative with respect to $\alpha$, and the equality holds if the zero-Hamiltonian 
constraint \eqref{eqn:H_constraint} {is satisfied}. With respect to $\alpha$, the zero-Hamiltonian constraint can be written as
\begin{equation}
|{\tilde{\phi}}'|\dot{\alpha}(t)=|b({\tilde{\phi}})|,
\end{equation}
{from which we have}
\begin{equation}
t=\int_0^{\alpha(t)}\frac{|{\tilde{\phi}}'|}{|b({\tilde{\phi}})|}d\alpha.
\end{equation}
If $|{\tilde{\phi}}'|\equiv cnst$, the variable $\alpha$ is nothing but a rescaled arc
length. Assuming that the length of the optimal curve is finite, we can rescale
the total arc length to one, i.e., $\alpha(T)=1$, which yields the geometric
minimum action method (gMAM) \cite{Heymann_CPAM08}.

We now look at any transition path $\tilde{\phi}(\alpha)=\phi(t(\alpha))$ that
satisfies the zero-Hamiltonian constraint. Let  $\alpha$
corresponds to the arc length {with $\tilde\phi(0)=0$, then $|\tilde{\phi}'|=1$}. 
Let $y$ be an arbitrary point on $\tilde{\phi}$.
Then the integration time from $0$ to $y$ is
\[
t=\int_0^{\alpha_y}\frac{1}{|b(\tilde{\phi})|}d\alpha,
\]
where $\alpha_y$ is the arc length of the curve connecting $0$ and $y${, i.e., the 
	value of the arc length variable $\alpha(t)$ at point $y$}. Note that if {the end 
	point $y$ is in a small neighborhood of the equilibrium $0$, the total arc length 
	from $0$ to $y$ along $\bar{\phi}$ is small}. {However, from the fact}
\[
|b(\tilde{\phi})|=|b(\tilde{\phi})-b(0)|\leq K|\tilde{\phi}|\leq K\alpha,
\] 
{we get}
\[
t=\int_0^{\alpha_y}\frac{1}{|b(\tilde{\phi})|}d\alpha\geq
\int_0^{\alpha_y}\frac{1}{K\alpha}d\alpha=\infty,
\]
as long as $\alpha_y>0$. So $T^*=\infty$, because $0$ is a critical point.

For clarity, we include the starting and ending points of the transition
path into some notations. Let $\phi^*_{y,x}$ indicate the minimizer of
Problem II with starting point $y$ and ending point $x$, and let
$T^*_{y,x}$ be the corresponding optimal integration time. We have
for any $y$ on $\phi_{0,x}^*$, $T_{0,y}^*=\infty$ and
$T_{y,x}^*<\infty$ for the exit problem as long as the $\phi_{y,x}^*$ has
a finite length.

\subsubsection{Minimizing sequence {for} $\hat{S}(\phi)$}\label{sec:min_sequence_T_inf}
Let
$\phi^L_{0,y}=y t$ be the linear function connecting 0 and $y$ in one
time unit {$T=1$}. Then
\begin{align*}
S_{T^*_{0,y}}(\phi^*_{0,y})\leq S_T(\phi^L_{0,y})&=
\frac{1}{2}\int_0^1\left|y-b(yt)\right|^2\,dt\\
&\leq\int_0^1(|y|^2+K^2|y|^2t^2)\,dt\leq C(K)\rho^2,
\end{align*}
where $|y|\leq \rho$. Although $T^*_{0,y}=\infty$ for any finite $\rho$, the action
$S_{T^*_{0,y}}(\phi^*_{0,y})$ decreases with respect to $\rho$ to zero.
We consider a sequence of optimization problems
\begin{equation}\label{eqn:linear_scaling_constraint}
\hat{S}(\bar{\phi}^{*,n}_{0,x})=
\inf_{\substack{\hat{T}(\bar{\phi})\leq n,\\
\bar{\phi}\in\mathcal{A}_1}} \hat{S}(\bar{\phi}), \quad n=1,2,3,\ldots
\end{equation}
generated by the constraint $\hat{T}(\bar{\phi})\leq n$. We have that
\begin{lemma}\label{lem:min_sequence}
$\{\bar{\phi}_{0,x}^{*,n}\}_{n=1}^\infty$ is a minimizing sequence{ of \eqref{problem_reformulation}}.
\end{lemma}
\begin{proof}
First of all, $\hat{S}(\phi^{*,n}_{0,x})$
is decreasing as $n$ increases. Pick one $\rho$ such that $x\notin
{B_{\rho}(0)}$ and consider a sequence of $\rho_k=2^{-k}\rho$, $k=1,2,
\ldots$. Let 
$y_k$ be the first intersection point of $\phi_{0,x}^*$ and
{$B_{\rho_k}(0)$} when traveling along $\phi_{0,x}^*$ from $x$ to $0${, so $|y_k|=\rho_k$}. The optimal transition
time $T^*_{y_k,x}<\infty$. We construct a path from $0$ to $x$ as
follows:
\[
\phi_k=\left\{
\begin{array}{ll}
\phi^L_{0,y_k}, t\in[-T^*_{y_k,x}-1,-T^*_{y_k,x}],\\
\phi^*_{y_k,x},t\in[-T^*_{y_k,x},0].
\end{array}
\right.
\]
Due to the additivity, we know that $\phi_{y_k,x}^{*}$ is located on
$\phi_{0,x}^*$ since $y_k\in\phi_{0,x}^*$. Then $\{(T_k=T_{y_k,x}
^{*}+1,\phi_k)\}$ is a minimizing sequence as {$\rho_k$} decreases, and
\[
S_{T_k}(\phi_k)\leq S_{T^*_{0,x}}
(\phi^*_{0,x})+C(K)\rho_{{k}}^2.
\]
Consider $n=\lceil T_k\rceil$. We have
\[
\hat{S}(\phi_{0,x}^{*,n})\leq S_{T_k}(\phi_k)\leq
S_{T^*_{0,x}}{(\phi^*_{0,x})}+C(K)\rho_{{k}}^2,
\]
where the first inequality is because $\hat{T}(\phi_k)$ is, in general, not
equal to $T_k$.
We then reach the conclusion.
\end{proof}

When $T^*=\infty$, we  have $\hat{T}(\bar{\phi})\rightarrow\infty$ as
$\bar{\phi}$ goes close to the minimizer, implying that
$|\bar{\phi}'|_{0,\Gamma_1}\rightarrow\infty$. Thus for this case, we
cannot study the convergence in $H^1(\Gamma_1)$. A larger space is
needed, i.e., the space consisting of absolutely continuous functions. {So we use $\bar{C}_{x_1}^{x_2}(0,T)$, the space of absolutely continuous functions connecting 
	$x_1$ and $x_2$ on $[0,T]$ with $0<T\leq\infty$. To prove the convergence of the 
	minimizing sequence in Lemma \ref {lem:min_sequence}, we use the following lemma, 
	which is Proposition 2.1 proved in \cite{Heymann_CPAM08}}:

\begin{lemma}\label{lemma_gmam}
Assume that the sequence $\big((T_k,\phi_k)\big)_{k\in\mathbb{N}}$ with $T_k>0$ and $\phi_k\in\bar{C}_{x_1}^{x_2}(0,T_k)$ for every $k\in\mathbb{N}$, is a minimizing 
sequence of \eqref{eqn:quasi-potential} and that the lengths of the curves of $\phi_k$ 
are uniformly bounded, i.e.
\[
\lim_{k\rightarrow\infty}S_{T_k}(\phi_k)=V(x_1,x_2)\quad\text{and}
\quad\sup_{k\in\mathbb{N}}\int_0^{T_k}|\dot{\phi}_{k}(t)|\,dt<\infty.
\]
Then the the action functional $\hat{S}$ has a minimizer $\varphi^*$, and for some 
subsequence $(\phi_{k_l})_{l\in\mathbb{N}}$ we have that
\[
\lim_{l\rightarrow\infty}{d}(\phi_{k_l},\varphi^*)=0,
\]
where ${d}$ denotes the Fr\'echet distance.
\end{lemma}

\begin{theorem}\label{thm:gmam}
 Assume that the lengths of the curves $\phi_{0,x}^{*,n}$ 
are uniformly bounded. Then there exists a subsequence $\phi_{0,x}^{*,n_l}$ 
that converges to a minimizer
$\phi^*\in \bar{C}_{0}^{x}(0,T)$ with respect to Fr\'{e}chet distance.
\end{theorem}
\begin{proof}
	{By Lemma \ref{lem:opt_linear_scaling}, we have a one-to-one correspondence between $\{\bar{\phi}_{0,x}^{*,n}\}_{n=1}^\infty$ 
		and $\{\phi_{0,x}^{*,n}\}_{n=1}^\infty$.} {So from} Lemma \ref{lem:min_sequence}, we know that $\{{(n,\phi_{0,x}^{*,n})}\}_{n=1}^\infty$ 
	defines a minimizing sequence{ of Problem II}. The convergence is a direct application of  
	Lemma \ref{lemma_gmam}. 
\end{proof}

Although we just
constructed $\{\phi\}_{0,x}^{*,n}$ for an exit problem around the neighborhood
of an asymptotically stable equilibrium, it is easy to see that the idea can
be applied to a global transition, say both $x_1$ and $x_2$ are asymptotically
stable equilibrium, as long as there exist finitely many critical points along
the minimal action path.
In equation \eqref{eqn:linear_scaling_constraint}, we {introduced} an extra
constraint $\hat{T}(\bar{\phi})\leq n$, which implies that the infimum may be
reached at the boundary $\hat{T}(\bar{\phi})=n$. From the optimization
point of view, such a box-type constraint is not favorable. Next, we will
show that this constraint is not needed for the discrete problem.

\subsubsection{Remove the constraint $\hat{T}(\bar{\phi})\leq n$ for a discrete
problem}
The key observation is as follows:
\begin{lemma}\label{lem:T_opt_discrete}
If $\bar{\phi}_h^*$ is the minimizer of $\hat{S}_h(\bar{\phi}_h)$ over
$\overline{\mathcal{B}}_h$, then $\hat{T}(\bar{\phi}_h^*){\le C_h}<\infty$, for {any given} $h$.
\end{lemma}
\begin{proof}
We argue by contradiction. Note that
\[
\hat{T}_{{}}^2(\bar{\phi}_h^*)=\frac{|(\bar{\phi}_h^*)'|^2_{0,\Gamma_1}}
{| b(\bar{\phi}_h^*)|^2_{0,\Gamma_1}}
\leq C\frac{|\bar{\phi}_h^*|^2_{0,\Gamma_1}}
{|b(\bar{\phi}_h^*)|^2_{0,\Gamma_1}},
\]
where the last inequality is from the inverse inequality of finite
element discretization and the constant $C$ only depends on mesh
\cite{Ciarlet02}. If
$\hat{T}_{{}}(\bar{\phi}_h^*)=\infty$, we have two possible cases:
$|b(\bar{\phi}_h^*)|_{0,\Gamma_1}=0$ or
$|\bar{\phi}_h^*|_{0,\Gamma_1}=\infty$. The first case implies that {$b(\bar{\phi}_h^*(s))=0$ for all $s\in\Gamma_1$, which contradicts (3) of Assumption \ref{assump:1}}. The second case implies that $\bar{\phi}_h^*$ must go to infinity somewhere due to the continuity, which contradicts Lemma \ref{lem:for_assump_2}.
\end{proof}

Lemma \ref{lem:T_opt_discrete} means that for a discrete problem the
constraint $\hat{T}(\bar{\phi})\leq n$ in equation
\eqref{eqn:linear_scaling_constraint} is not necessary in the sense that
there always exists a {number} $n$ such that $\hat{T}(\bar{\phi}_h^*)<n$.
We can then consider a sequence $\{\overline{\mathcal{T}}_h\}$ of finite
element meshes and treat the minimization of $\hat{S}_h(\bar{\phi}_h)$ exactly
{in} the same way as the case that $T^*$ is finite. Simply speaking,
$\{(\hat{T}(\bar{\phi}_h^*),\bar{\phi}_h^*)\}$ defines a minimizing
sequence as $h\rightarrow0$ no matter that $T^*$ is finite or infinite.
The only difference between $T^*<\infty$ and $T^*=\infty$ is that we
address the convergence of $\bar{\phi}_h^*$ in $H^1(\Gamma_1)$ for $T^*<\infty$
and in $\bar{C}_{x_1}^{x_2}$ for $T^*=\infty$.

\subsubsection{The efficiency of $\{(\hat{T}(\bar{\phi}_h^*),
\bar{\phi}_h^*)\}$}
The {E-L} equation associated with $\hat{S}(\phi)$ is:
\begin{equation}\label{eqn:EL}
\hat{T}^{-2}(\bar{\phi})\bar{\phi}''+\hat{T}^{-1}(\bar{\phi})
\left((\nabla_{\bar{\phi}}b)^\mathsf{T}-\nabla_{\bar{\phi}}b\right)
\bar{\phi}'-(\nabla_{\bar{\phi}}b)^\mathsf{T}b=0.
\end{equation}
For a fixed $T$, the E-L equation associated with $S_T(\phi)$ is the same
as equation \eqref{eqn:EL} except that we need to replace $\hat{T}(\phi)$
with $T$ \cite{Wan_tMAM}. If $T^*=\infty$, $\hat{T}(\bar{\phi}_h^*)
\rightarrow\infty$ as $h\rightarrow0$, which means that the E-L equation
\eqref{eqn:EL} eventually becomes degenerate. When $h$ is small,
$\hat{T}(\phi_h^*)$ is finite but large. Equation \eqref{eqn:EL} can be
regarded as a singularly perturbed problem, which implies {the possible existence of} boundary/internal layers. Thus the minimizing sequence given by {a} uniform refinement  $\overline{\mathcal{T}}_h$ {of $\Gamma_1$} may not be effective.

Consider a transition from a stable fixed point to a saddle point without any
other critical points on the minimal transition path. The dynamics is slow
around the critical points and fast else where, which means that for the
approximation given by a fixed large $T$, the path will be mainly captured in
a subinterval $[a,b]\in\Gamma_1$ with $|b-a|\sim\mathit{O}(T^{-1})$ with
respect to the scaled time $s=t/T$. Then an effective
$\overline{\mathcal{T}}_h$ {has} a fine mesh for $[a,b]$ and a coarse mesh
for $[0,a]\cup[b,1]$. Currently, there exist two techniques to achieve an
effective nonuniform discretization $\overline{\mathcal{T}}_h$: (1) Moving
mesh technique. Starting from a fine uniform discretization, the grid points
will be redistributed such that more grid points will be moved from the
region of slow dynamics to the region of fast dynamics
\cite{Zhou_JCP08,Wan_MAM11}. This procedure needs to be iterated until the
optimal nonuniform mesh is reached with respect to a certain criterion
for the redistribution of grids. (2) Adaptive finite element method.
Starting from a coarse uniform mesh, $\overline{\mathcal{T}}_h$ will be
refined adaptively such that more elements will be put into the region
of fast dynamics and less elements into the region of slow dynamics
\cite{Wan_tMAM,Wan_tMAM_hp}. Numerical experiments have
shown that both techniques can recover the optimal convergence rate with
respect to the number of degrees of freedom.

\section{A priori error estimate for a linear ODE system}\label{sec:error}
In this section we apply our strategy to a linear ODE system
with $b(x)=-Ax$, where $A$ is {a} symmetric positive definite
matrix. Then $x=0$ is a global attractor. The E-L equation
associated with $\hat{S}(\bar{\phi})$ becomes
\begin{equation}\label{eqn:EL_simple}
-\hat{T}^{-2}(\bar{\phi})\bar{\phi}''+A^2\bar{\phi}=0,
\end{equation}
which is a nonlocal elliptic problem of Kirchhoff type. For
Problem I with a fixed $T$, i.e., $\hat{T}(\phi)=T$, equation
\eqref{eqn:EL_simple} becomes a standard diffusion-reaction
equation.

Let
$\bar{\phi}^*=\bar{\phi}_0^*+\phi_L\in \mathcal{A}_1$
be the minimizer of $\hat{S}(\bar{\phi})$, and
$\bar{\phi}_h^*=\bar{\phi}_{h,0}^*+
\phi_L\in\overline{\mathcal{B}}_h$ the
minimizer of $\hat{S}_h(\bar{\phi}_h)$, where
$\phi_L=x_1+{(x_2-x_1)}s$ is a linear function connecting $x_1$ and
$x_2$ on $\Gamma_1$. Let
\[
V_h =\{v{\,:\:}\,v|_I\textrm{ is affine for all }I\in
\overline{\mathcal{T}}_h,\,v(0)=v(1)=0\}\subset H_0^1(\Gamma_1).
\]
For a fixed $T$, equation \eqref{eqn:EL_simple} has a unique
solution and the standard argument {shows} that
\begin{equation}\label{eqn:prior_error_fixed_T}
\left|\bar{\phi}^*-\bar{\phi}^*_{h}\right|_{1,\Gamma_1}
=\left|\bar{\phi}^*_0-\bar{\phi}^*_{h,0}\right|_{1,\Gamma_{{1}}}
\leq CT^2\inf_{w\in V_h}\left|\bar{\phi}^*_0-
w\right|_{1,\Gamma_1}.
\end{equation}
If $T$ is large enough, equation \eqref{eqn:EL_simple} can be
regarded as a singularly perturbed problem, the best
approximation given by a uniform mesh cannot reach the optimal
convergence rate due to the existence of boundary layer.

We now consider Problem II with a finite $T^*$. The minimizer
$\bar{\phi}^*$ of $\hat{S}(\bar{\phi})$ satisfies the weak form
of equation \eqref{eqn:EL_simple}:
\begin{equation}\label{eqn:weak_c1}
\left\langle \left(\bar{\phi}^*\right)',
v'\right\rangle_{\Gamma_1}=-\hat{T}^{2}\left(\bar{\phi}^*\right)
\left\langle A\bar{\phi}^*,Av\right\rangle_{\Gamma_1},
\quad\forall\ v\in H_0^1(\Gamma_1).
\end{equation}
The minimizer $\bar{\phi}_h^*$ of $\hat{S}_h(\bar{\phi}_h)$
satisfies the discrete weak form:
\begin{equation}\label{eqn:weak_c2}
\left\langle
\left(\bar{\phi}^*_h\right)',v'\right\rangle_{\Gamma_1}=
-\hat{T}^{2}\left(\bar{\phi}^*_h\right)\left\langle
A\bar{\phi}_h^*,Av\right\rangle_{\Gamma_1},
\quad\forall\ v\in V_h{.}
\end{equation}
We have the following {a} priori error estimate for
Problem II with a finite $T^*$:
\begin{proposition}\label{lem:priori}
Consider a subsequence $\bar{\phi}_h^*$ converging weakly to $
\bar{\phi}^*$ in $H^1(\Gamma_1)$ {as $h\rightarrow0$}. Assume
that $\bar{\phi}^*$ and $\bar{\phi}_h^*$ satisfy equations
\eqref{eqn:weak_c1} and \eqref{eqn:weak_c2}, respectively. For
problem II with a finite $T^*$, there exists a
constant $C\sim(T^*)^2$ such that
\begin{equation}
\left|\bar{\phi}^*-\bar{\phi}^*_{h}\right|_{1,\Gamma_1}
=\left|\bar{\phi}^*_0-\bar{\phi}^*_{h,0}\right|_{1,\Gamma_1}
\leq C \inf_{w\in V_h}\left|\bar{\phi}^*_0-
w\right|_{1,\Gamma_1},
\end{equation}
when $h$ is small enough.
\end{proposition}
\begin{proof}
Let $\eta$ be the best approximation of $\bar{\phi}^*$ on
$V_h\oplus\phi_L$, i.e.,
\[
\left|\bar{\phi}^*-\eta\right|_{1,\Gamma_1}=
\inf_{w\in V_h\oplus\phi_L}
\left|\bar{\phi}^*-w\right|_{1,\Gamma_1}.
\]
We then have
\[
\langle(\bar{\phi}^*-\eta)',w'\rangle=0,\quad \forall\ w\in V_h,
\]
where $\bar{\phi}^*-\eta\in H_0^1(\Gamma_1)$. Consider
\begin{align}
\left|\bar{\phi}^*_h-\eta\right|^2_{1,\Gamma_1}=&\left\langle
(\bar{\phi}_h^*-\bar{\phi}^*)',(\bar{\phi}_h^*-
\eta)'\right\rangle_{\Gamma_1}+\left\langle
(\bar{\phi}^*-\eta)',(\bar{\phi}_h^*-
\eta)'\right\rangle_{\Gamma_1}\nonumber\\
=&\left\langle
(\bar{\phi}_h^*-\bar{\phi}^*)',(\bar{\phi}_h^*-
\eta)'\right\rangle_{\Gamma_1}\nonumber\\
=&-\left\langle
(\hat{T}^2(\bar{\phi}_h^*)\bar{\phi}_h^*-\hat{T}^2(\bar{\phi}^*)
\bar{\phi}^*),A^2(\bar{\phi}^*_h-\eta)
\right\rangle_{\Gamma_1}\nonumber\\
=&-\left\langle
(\hat{T}^2(\bar{\phi}_h^*)\bar{\phi}_h^*-\hat{T}^2(\eta)
\eta),A^2(\bar{\phi}^*_h-\eta)
\right\rangle_{\Gamma_1}\nonumber\\
&-\left\langle
(\hat{T}^2(\eta)\eta-\hat{T}^2(\bar{\phi}^*)
\bar{\phi}^*),A^2(\bar{\phi}^*_h-\eta)
\right\rangle_{\Gamma_1}\nonumber\\
=&I_1+I_2.
\label{eqn:priori_1}
\end{align}
We look {at} $I_2$ first. Note that
\begin{align}
&|\hat{T}^2(\eta)-\hat{T}^2(\bar{\phi}^*)|\nonumber\\
&=
\left|\frac{|\eta|^2_{1,\Gamma_1}}{|A\eta|^2_{0,\Gamma_1}}-
\frac{|\bar{\phi}^*|^2_{1,\Gamma_1}}{|A\bar{\phi}^*|^2_{0,\Gamma_1}}\right|\nonumber\\
&=\left|\frac{|\eta|^2_{1,\Gamma_1}-|\bar{\phi}^*|^2_{1,\Gamma_1}
}{|A\eta|^2_{0,\Gamma_1}}+\frac{|\bar{\phi}^*|^2_{1,\Gamma_1}
(|A\bar{\phi}^*|^2_{0,\Gamma_1}-|A\eta|^2_{0,\Gamma_1})}{|A\eta|^2
_{0,\Gamma_1}|A\bar{\phi}^*|^2_{0,\Gamma_1}}\right|\nonumber\\
&\leq C_{\hat{T}}(\eta,\bar{\phi}^*)|\eta-\bar{\phi}^*|_{1,\Gamma_1},
\label{Error_Estimate_T}
\end{align}
where
\[
C_{\hat{T}}(\eta,\bar{\phi}^*)=
\frac{|\eta|_{1,\Gamma_1}+|\bar{\phi}^*|_{1,\Gamma_1}}{|A\eta|^2_{0,\Gamma_1}}
+\frac{|\bar{\phi}^*|^2_{1,\Gamma_1}(|A\eta|_{0,\Gamma_1}+
|A\bar{\phi}^*|_{0,\Gamma_1})
\|A\|C_p}{|A\eta|^2_{0,\Gamma_1}|A\bar{\phi}^*|^2_{0,\Gamma_1}},
\]
and $C_p$ is the Poincar\'{e} constant. Then we have
\begin{align}
|I_2|=&\left|\left\langle
(\hat{T}^2(\eta)\eta-\hat{T}^2(\bar{\phi}^*)
\bar{\phi}^*),A^2(\bar{\phi}^*_h-\eta)
\right\rangle_{\Gamma_1}\right|\nonumber\\
\leq&\left|\left\langle
\hat{T}^2(\eta)(\eta-
\bar{\phi}^*),A^2(\bar{\phi}^*_h-\eta)
\right\rangle_{\Gamma_1}\right|\nonumber\\
&+\left|\left\langle
(\hat{T}^2(\eta)-\hat{T}^2(\bar{\phi}^*))
\bar{\phi}^*,A^2(\bar{\phi}^*_h-\eta)
\right\rangle_{\Gamma_1}\right|\nonumber\\
\leq&(\hat{T}^2(\eta)C_p^2\|A\|^2+
C_{\hat{T}}(\eta,\bar{\phi}^*)
|A\bar{\phi}^*|_{0,\Gamma_1}\|A\|C_p)|\eta-\bar{\phi}^*|_{1,\Gamma_1}
|\bar{\phi}_h^*-\eta|_{1,\Gamma_1}
\nonumber\\
=&C_{I_2}(\eta,\bar{\phi}^*)|\eta-\bar{\phi}^*|_{1,\Gamma_1}|
\bar{\phi}_h^*-\eta|_{1,\Gamma_1}.
\label{eqn:priori_2}
\end{align}
By the definition of $\eta$, we have
\[
\lim_{h\rightarrow0}|\eta|_{1,\Gamma_1}=|\bar{\phi}^*|_{1,\Gamma_1},
\quad \lim_{h\rightarrow0}|A\eta|_{0,\Gamma_1}=|A\bar{\phi}^*|_{0,\Gamma_1}.
\]
We then have
\[
\lim_{h\rightarrow0}C_{I_2}(\eta,\bar{\phi}^*)=2M+3M^2,
\]
where $M=\|A\|C_pT^*$.

We now look at $I_1$. Since $\lim_{h\rightarrow0}\hat{T}
(\bar{\phi}_h^*)=\lim_{h\rightarrow0}\hat{T}(\eta)=T^*$ and
$T^*<\infty$, we know that when $h$ is small enough, $I_1\sim
-(T^*)^2\left\langle
(\bar{\phi}_h^*-
\eta),A^2(\bar{\phi}^*_h-\eta)
\right\rangle_{\Gamma_1}<0$.
Combining this fact with equations \eqref{eqn:priori_1} and
\eqref{eqn:priori_2}, we have that for $h$ small enough there
exists a constant $C>2M+3M^2$ such that
\[
\left|\bar{\phi}^*_h-\eta\right|_{1,\Gamma_1}
\leq C\left|\bar{\phi}^*-
\eta\right|_{1,\Gamma_1}
=C\inf_{w\in V_h\oplus\phi_L}\left|\bar{\phi}^*-
w\right|_{1,\Gamma_1}.
\]
\end{proof}


To this end, we obtain a similar a priori error estimate to that
for Problem I with a fixed $T$. Since $T^*$ can be arbitrarily
large, we know that the optimal convergence rate may degenerate
when a boundary layer exists. Using {Proposition} \ref{lem:priori}, we
can easily obtain the optimal convergence rate with respect to
the error of action functional:
\begin{equation}\label{eqn:S_cong}
|\hat{S}(\bar{\phi}^*)-\hat{S}(\bar{\phi}_h^*)|
\sim|\delta^2\hat{S}(\bar{\phi}^*)|\sim
|\bar{\phi}_h^*-\bar{\phi}^*|_{1,\Gamma_1}^2,
\end{equation}
where the second-order variation can be obtained with respect
to the perturbation function $\delta\bar{\phi}=\bar{\phi}^*
-\bar{\phi}_h^*$. 

{The convergence rate for $T^*$ is also optimal. For a 
	general case, we have  the first-order variation of $\hat{T}^2$ at $\bar{\phi}^*$ with a test function 
	$\delta\bar{\phi}$ as
	\begin{align*}
	|\delta(\hat{T}^2)|&=\left|\frac{2\langle (\bar{\phi}^*)'',\delta\bar{\phi}\rangle_{\Gamma_1}+2(T^*)^2\langle (\nabla b)^\mathsf{T}b,\delta\bar{\phi}\rangle_{\Gamma_1}}{|b(\bar{\phi}^*)|^2_{0,\Gamma_1}}\right|\\
	&\leq 2|b(\bar{\phi}^*)|^{-2}_{0,\Gamma_1}\left(
	|\bar{\phi}^*|_{2,\Gamma_1}+(T^*)^2K|b(\bar{\phi}^*)|_{0,\Gamma_1}
	\right)|\delta\phi|_{0,\Gamma_1}.
	\end{align*}
	Also note that the second order variation has the same order as $|\delta\phi|^2_{1,\Gamma_1}$. Let  $\delta\phi=\bar{\phi}_h^*-\bar{\phi}^*$. From Proposition \ref{lem:priori}, we know the second-order variation 
	has an optimal convergence rate $|\bar{\phi}_h^*-\bar{\phi}^*|_{1,\Gamma_1}^2\sim\mathit{O}(h^2)$, when $\bar{\phi}^*\in H^2(\Gamma_1)$. If $|\bar{\phi}_h^*-\bar{\phi}^*|_{0,\Gamma_1}$ can also reach its optimal rate, which is 
	of order $\mathit{O}(h^2)$,  the overall convergence rate for $T^*$ is of order $\mathit{O}(h^2)$. We will not analyze the optimal convergence of $\bar{\phi}_h^*$ in $L^2$ norm here, but only 	
	provide numerical evidence of the optimal convergence rate for $T^*$ in the next section.
	
}

\section{Numerical experiments}
We will use the following simple linear {stochastic ODE} system to demonstrate our analysis results
\begin{equation}\label{eqn:numerical_example}
dX(t)=AX(t)\,dt+\sqrt{\varepsilon}\,dW(t),
\end{equation}
where
\[
A=B^{-1}JB=\begin{bmatrix}
a & -b \\ b & a \end{bmatrix}\begin{bmatrix}
\lambda_1 & 0 \\ 0 & \lambda_2 \end{bmatrix}\begin{bmatrix}
a & b \\ -b & a \end{bmatrix},
\]
with $a=1/3$, $b=\sqrt{8}/3$, $\lambda_1=-10$, and $\lambda_2=-2$. Then $z=(0,0)^{\mathsf{T}}$ is a stable fixed point. For the corresponding deterministic system, namely when $\varepsilon=0$, and any given point $X(0)=x\not=z$ in the phase space, the  trajectory $X(t)=e^{tA}x$ converges to $z$ as $t\rightarrow\infty$. When noise exists,  this trajectory is also the {minimal} action path $\phi^*$ from $x$ to $e^{tA}x$ with
$T^*=t$, since $V(x,e^{tA}x)=0$. Moreover, if the ending point is $z$,
$T^*=\infty$. This obviously is not an exit problem, which is {a} typical application
of MAM. However, it includes most of the numerical difficulties of MAM, and the trajectory
can serve as an exact solution, which simplifies the discussions.

Consider the minimal action path from $x$ ($\neq z$) to $e^{tA}x$ such that $T^*=t$. Since the minimal action path corresponds to a trajectory, we can use the value of action functional as the measure of error with an optimal rate $\mathit{O}(h^2)\sim\mathit{O}(N^{-2})$ (see equation
\eqref{eqn:S_cong}), where $N$ is the number of elements.
We will look at the following two cases:
\begin{romannum}
	\item $T^*$ is finite and small. According to Theorem \ref{thm:cong_minima_Shat}, $\bar{\phi}_h$ {converges} to $\bar{\phi}^*$. Since $T^*$ is small, according to
	Proposition \ref{lem:priori}, we expect optimal convergence rate of $\bar{\phi}_h$ as $h\rightarrow0$.
	\item $T^*=\infty$. We will compare the convergence behavior between tMAM and MAM with a fixed large $T$. According to Theorem \ref{thm:gmam}, we  have the convergence
	in $\bar{C}_x^z$. However, the convergence behavior of this case {is} similar to
	that for a finite but large $T^*$, where we expect a deteriorated convergence rate.
\end{romannum}

\subsection{Case (i)} Let $x=(1,1)$. We use $e^{A}x$ as the ending point such that $T^*=1$. In figure \ref{fig:T1} we plot the convergence behavior of tMAM with uniform
linear finite element discretization. It is seen that the optimal convergence rate
is reached for both action functional and $T^*$ estimated by $\hat{T}({\bar{\phi}}_h^*)$.

\begin{figure}
	\center{
		\includegraphics[width=0.49\textwidth]
		{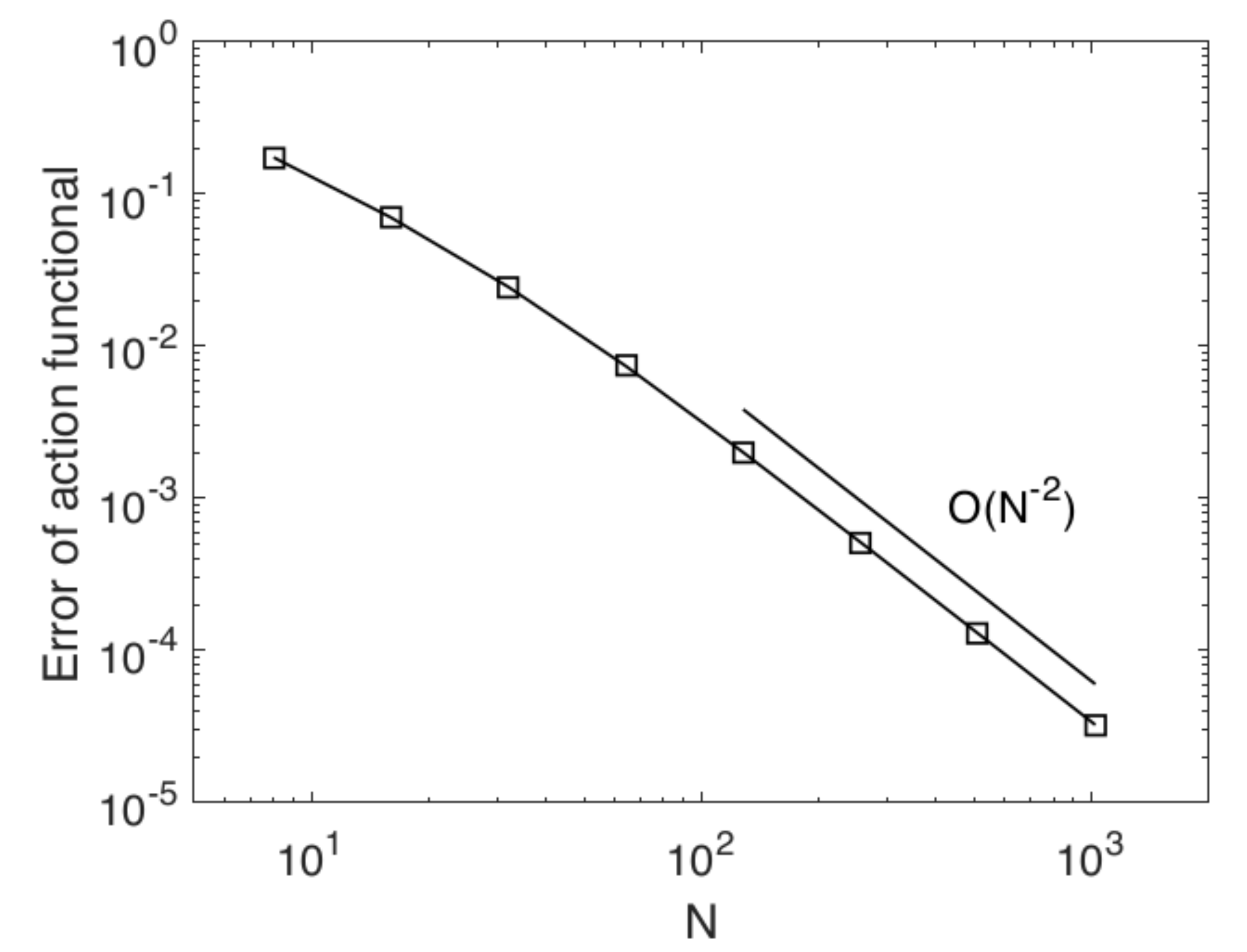}
		\includegraphics[width=0.49\textwidth]
		{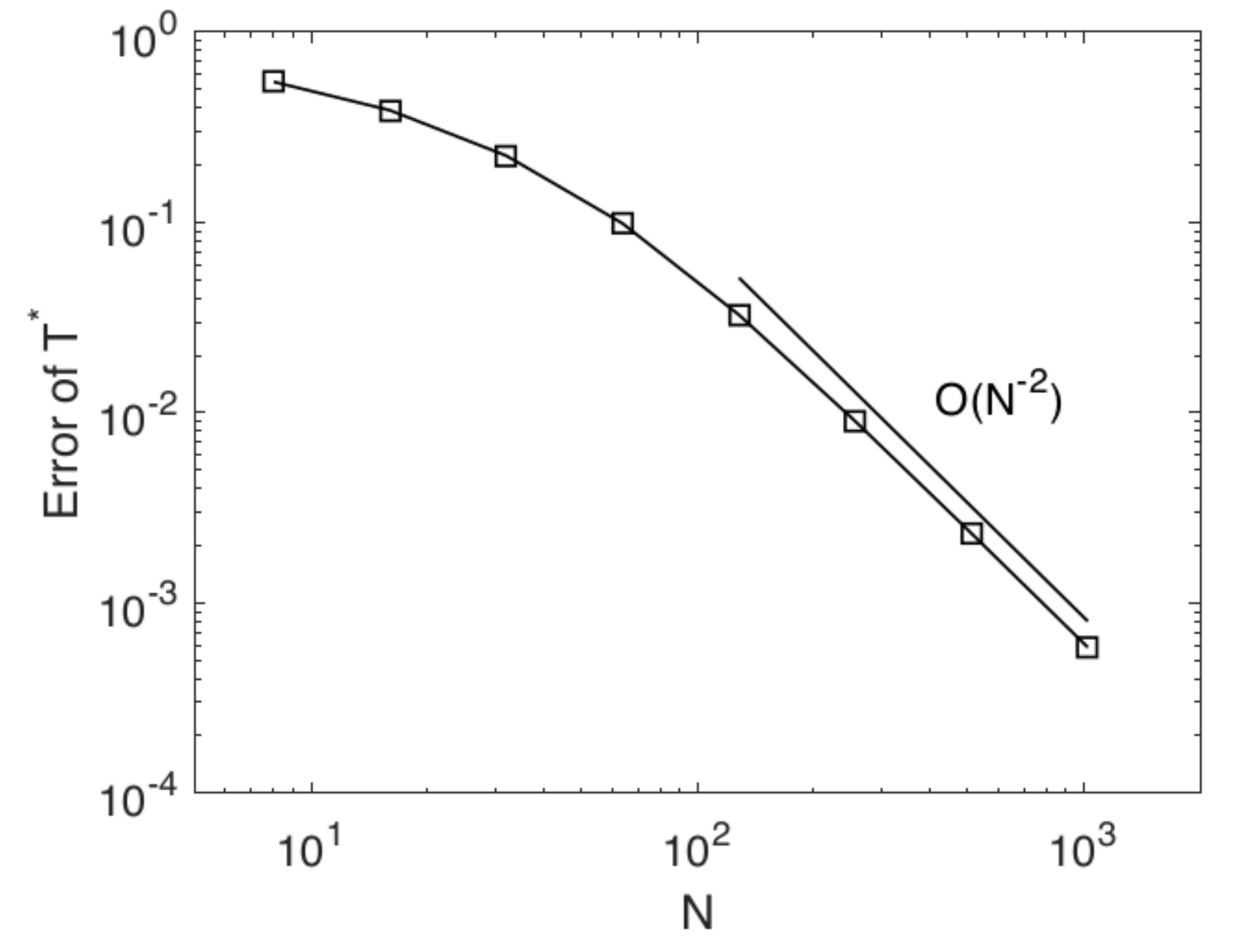}
	}
	\caption{Convergence behavior of tMAM for Case (i). Left: errors of action functional; Right: errors of $T^*$.}\label{fig:T1}
\end{figure}

\begin{figure}
	\center{
		\includegraphics[width=0.49\textwidth]
		{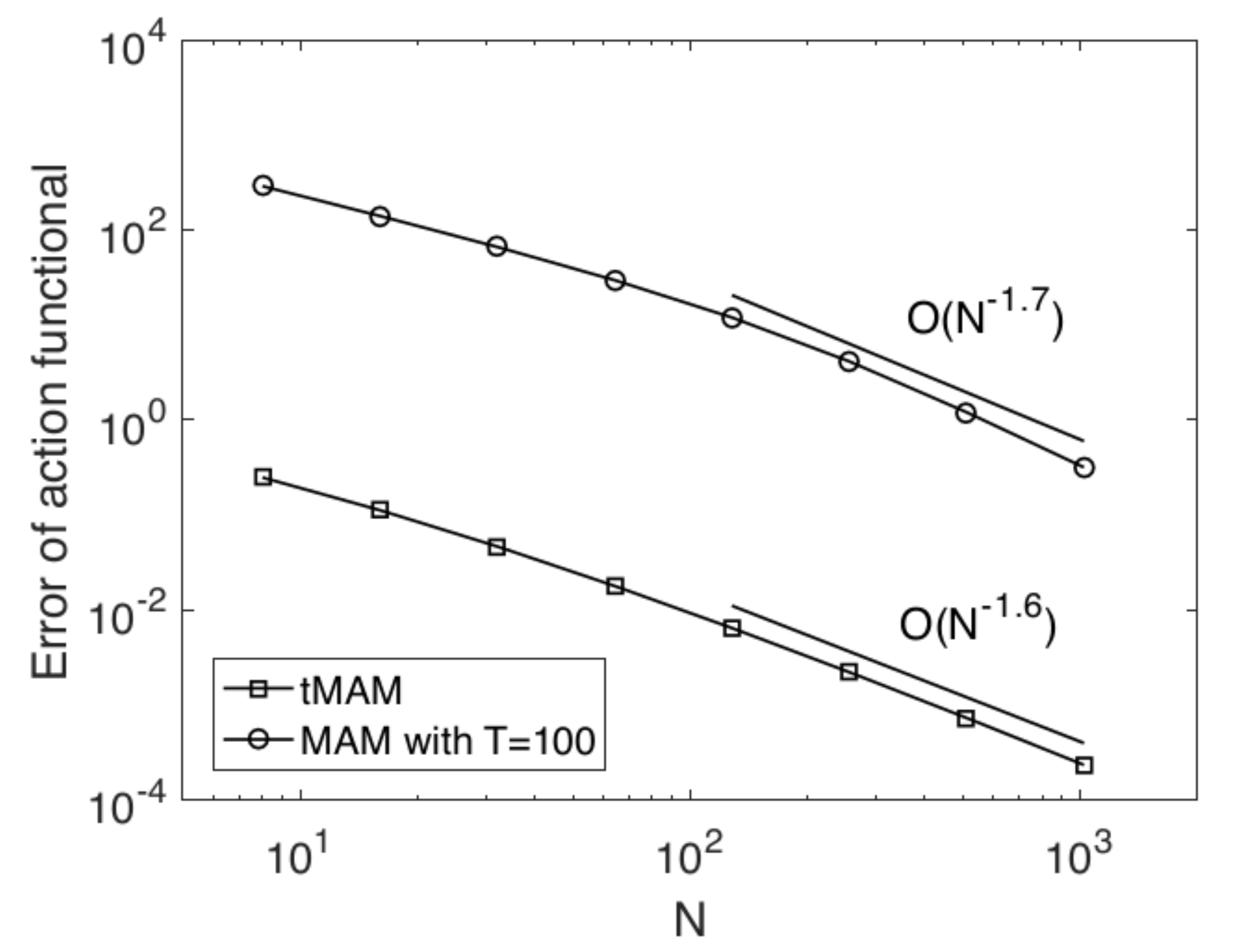}
		\includegraphics[width=0.49\textwidth]
		{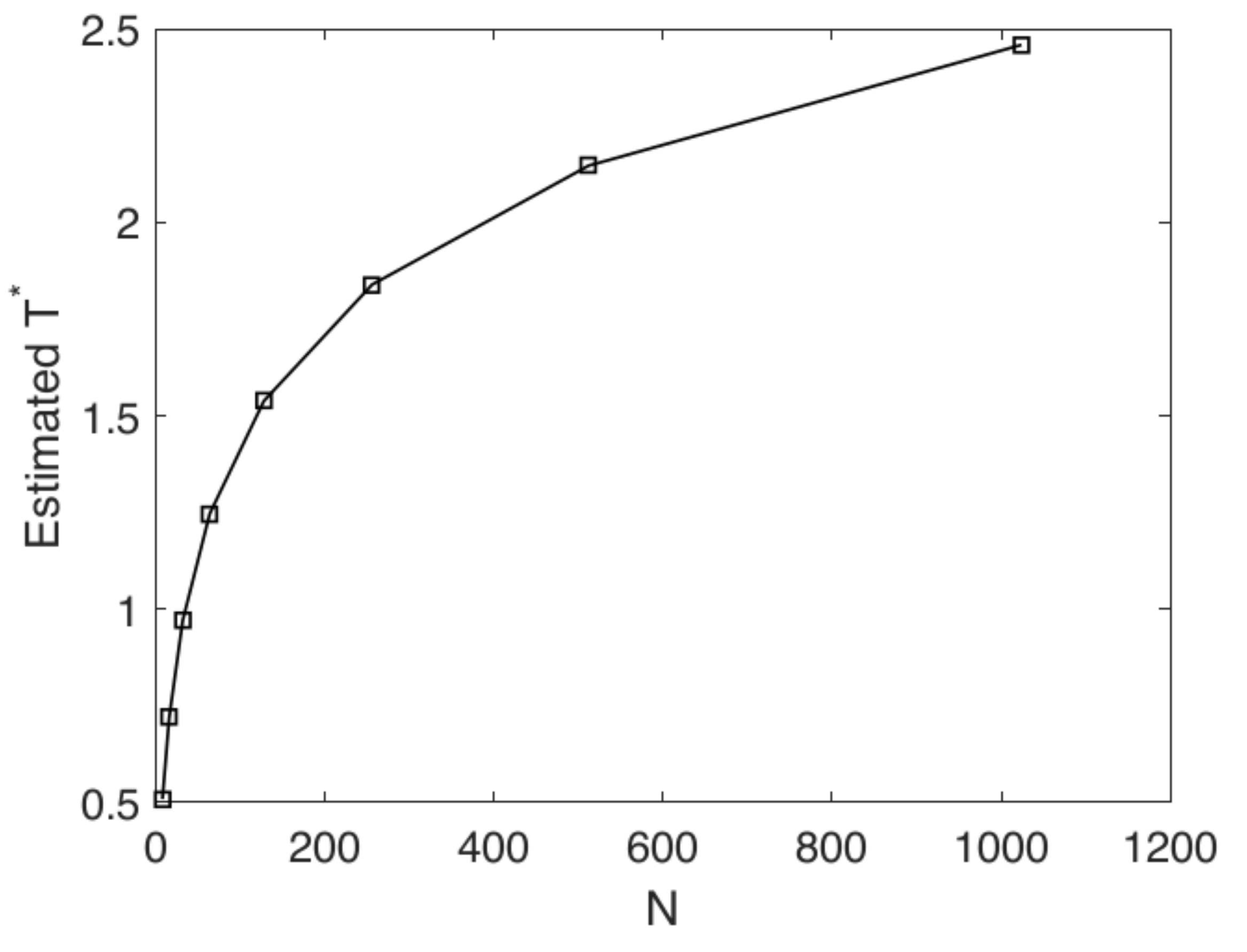}
	}
	\caption{Convergence behavior of tMAM and MAM with a fixed $T$ for Case (ii). Left: errors 
		of action functional; Right: estimated $T^*$ of tMAM, i.e., $\hat{T}(\bar{\phi}_h^*)$.}\label{fig:T_inf}
\end{figure}

\begin{figure}
	\center{
		\includegraphics[width=0.6\textwidth]
		{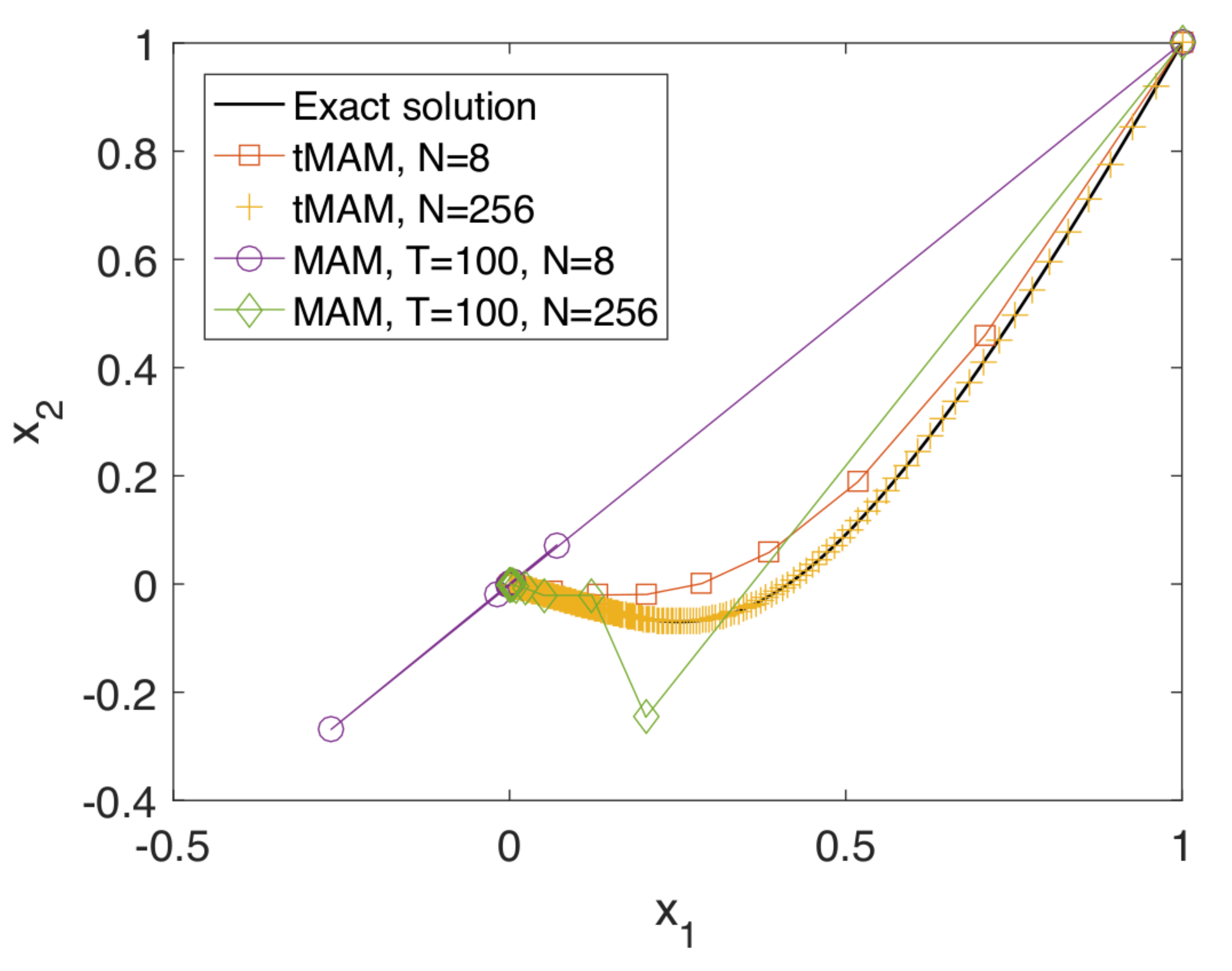}
	}
	\caption{Approximate MAPs given by tMAM and MAM with a fixed $T$ for Case (ii).}\label{fig:path}
\end{figure}

\subsection{Case (ii)} For this case, we still use $x=(1,1)$ as the starting point.
The ending point is chosen as $a=(0,0)^\mathsf{T}$ such that $T^*=\infty$. Except tMAM,
we use MAM with a fixed $T$ to approximate this case, where $T$ is supposed to be large.
In general, we do not have a criterion to define how large is enough because the accuracy
is affected by two competing issues: 1) The fact that $T^*=\infty$ favors a large $T$; but
2) a fixed discretization favors a small $T$. This implies that {for any given $h$,} 
an ``optimal'' finite $T$ exists. For the purpose of demonstration, we choose $T=100$, 
which is actually too large from the accuracy point of view. Let $\phi_T^*(t)$ be the 
approximate MAP given by MAM with a fixed $T$. We know that 
$\bar{\phi}^*_T(s)=\phi_T^*(t/T)$ {yields} a smaller action
with the integration time $\hat{T}(\bar{\phi}^*_T(s))$. In this sense, no matter what $T$
is chosen, for the same discretization tMAM will always provide a better approximation
than MAM with a fixed $T$. The reason we use an overlarge $T$ is to demonstrate the
deterioration of convergence rate. In figure \ref{fig:T_inf}, we plot the convergence 
behavior of tMAM and MAM with $T=100$ on the left, and the estimated $T^*$ given by tMAM
on the right. It is seen that the convergence is slower than $\mathit{O}(N^{-2})$ as 
we have analyzed in section \ref{sec:error}. For the same discretization,  tMAM has an
accuracy that is several orders of magnitude better than MAM with $T=100$.
In the right plot of figure \ref{fig:T_inf}, we
see that the optimal integration time for a certain discretization is actually not large
at all. This implies that MAM with a fixed $T$ for Case (ii) is actually not very reliable.
In figure \ref{fig:path}, we compare the MAPs given by tMAM and MAM with the exact
solution $e^{tA}x$, where all symbols indicate the nodes of finite element discretization.
First of all, we note that the number of effective nodes in MAM is small because of the
scale separation of fast dynamics and small dynamics. Most nodes are clustered around the
fixed point. This is called a problem of clustering (see \cite{Zhou_2016,Wan_tMAM_hp} 
for the discussion of this issue). Second, if the chosen $T$ is too large, oscillation 
is observed in the paths given by MAM especially when the resolution is relatively low; on 
the other hand, tMAM does not suffer such an oscillation by adjusting the integration 
time according to the resolution. 
Third, although tMAM is able to provide a good 
approximation even with a coarse discretization, more than enough nodes are put into the 
region around the fixed point, which corresponds to the deterioration of convergence rate. 
To recover the optimal convergence rate, we need to resort to adaptivity (see 
\cite{Wan_tMAM,Wan_tMAM_hp} for the construction of the algorithm).

\section{Summary}
In this work, we have established some convergence results of minimum action methods based
on linear finite element discretization. In particular, we have demonstrated that the
minimum action method with optimal linear time scaling, i.e., tMAM, converges for Problem
II no matter that the optimal integration time is finite or infinite.

\section*{Acknowledgement}
X. Wan and J. Zhai were supported by AFOSR grant FA9550-15-1-0051 and 
NSF Grant DMS-1620026. H. Yu was  supported by
NNSFC Grant 11771439,91530322 and Science Challenge Project No.~TZ2018001.

\end{document}